\documentclass[a4paper,11pt]{article}
\usepackage{amsfonts}
\usepackage{amssymb}
\usepackage{amsthm}
\usepackage{amsmath}
\usepackage{graphicx}
\usepackage{empheq}
\usepackage{indentfirst}
\usepackage{mathrsfs}
\usepackage{graphics}
\usepackage{color}

\newtheorem{theorem}{\textbf{Theorem}}[section]
\newtheorem{lemma}{\textbf{Lemma}}[section]
\newtheorem{proposition}{\textbf{Proposition}}[section]
\newtheorem{corollary}{\textbf{Corollary}}[section]
\newtheorem{remark}{\textbf{Remark}}[section]
\newtheorem{definition}{\textbf{Definition}}[section]

\allowdisplaybreaks[4]

\def\be{\begin{equation}}
\def\ee{\end{equation}}
\def\bea{\begin{eqnarray}}
\def\eea{\end{eqnarray}}
\def\bt{\begin{theorem}}
\def\et{\end{theorem}}
\def\bl{\begin{lemma}}
\def\el{\end{lemma}}
\def\br{\begin{remark}}
\def\er{\end{remark}}
\def\bp{\begin{proposition}}
\def\ep{\end{proposition}}
\def\bc{\begin{corollary}}
\def\ec{\end{corollary}}
\def\bd{\begin{definition}}
\def\ed{\end{definition}}

\def\vp{\varphi}

\def\non{\nonumber }
\def\ub{\mathbf{u}}
\def\btau{\boldsymbol{\tau}}
\def\BH{\mathbf{H}}

\title{Existence and uniqueness of global weak solutions to a Cahn-Hilliard-Stokes-Darcy system for two phase incompressible  flows in karstic geometry}

\author{
 {\sc Daozhi Han}\footnote{Department of Mathematics, Florida State University, Tallahassee, FL 32306, USA.
Email: \emph{dhan@math.fsu.edu}}\ \
 {\sc Xiaoming Wang}\footnote{ Department of Mathematics, Florida
State University, Tallahassee, FL 32306, USA. Email:
\emph{wxm@math.fsu.edu}}\ \
 and \ {\sc Hao Wu}\footnote{School of Mathematical Sciences and Shanghai Key Laboratory for Contemporary Applied
Mathematics, Fudan University,
Shanghai 200433, China. Email: \emph{haowufd@yahoo.com}}
 }

\date{\today}

\begin{document}
\maketitle
\begin{abstract}
We study the well-posedness of a coupled Cahn-Hilliard-Stokes-Darcy system which is a diffuse-interface  model for essentially immiscible two phase incompressible flows with matched density in a karstic geometry. Existence of finite energy weak solution that is global in time is established in both 2D and 3D. Weak-strong uniqueness property of the weak solutions is provided as well.
\end{abstract}




\section{Introduction}
\setcounter{equation}{0}

Applications such as contaminant transport in karst aquifer, oil recovery in karst oil reservoir, proton exchange membrane fuel cell technology and cardiovascular modelling require the coupling of flows in  conduits with those in the surrounding porous media. Geometric configurations that contain both conduit (or vug) and porous media are termed karstic geometry. Moreover, many flows are naturally multi-phase and hence multi-phase flows in the karstic geometry are of interest.
Despite the importance of the subject, little work has been done in this area. Our main goal here is to analyze a diffuse-interface model for two phase incompressible  flows with matched densities in the karstic geometry that was recently derived in \cite{HSW13} via Onsager's extremum principle.

To fix the notation, let us assume that
the two-phase flows are confined in a bounded connected domain $\Omega \subset \mathbb{R}^d$ ($d=2,3$) of $C^{2,1}$ boundary $\partial \Omega$. The unit outer normal at $\partial \Omega$ is denoted by $\mathbf{n}$. The domain $\Omega$ is partitioned into two non-overlapping regions such that $\overline{\Omega}=\overline{\Omega}_c\cup \overline{\Omega}_m$ and $\Omega_c \cap \Omega_m= \emptyset$, where $\Omega_c$ and $\Omega_m$ represent the underground
conduit (or vug) and the porous matrix region, respectively.
We denote $\partial \Omega_c$ and $\partial \Omega_m$ the boundaries of the conduit and the matrix part, respectively. Both $\partial \Omega_c$ and  $\partial \Omega_m$ are assumed to be Lipschitz continuous. The interface between the two parts (i.e., $\partial \Omega_c\cap \partial \Omega_m$) is denoted by $\Gamma_{cm}$, on which $\mathbf{n}_{cm}$ denotes the
unit normal to $\Gamma_{cm}$ pointing from the conduit part to the matrix part. Then we denote $\Gamma_c=\partial \Omega_c\backslash \Gamma_{cm}$ and $\Gamma_m=\partial \Omega_m\backslash \Gamma_{cm}$ with   $\mathbf{n}_c, \mathbf{n}_m$ being the unit outer normals to
$\Gamma_{c}$ and $\Gamma_{m}$. We assume that both $\Gamma_m$ and $\Gamma_{cm}$ have positive measure (namely, $|\Gamma_m|>0$, $|\Gamma_{cm}|>0$) but allow $\Gamma_c=\emptyset$, i.e. $\Omega_c$ can be enclosed completely by $\Omega_m$.  A two dimensional geometry is illustrated in Figure \ref{domain}. When $d=3$, we also assume that the surfaces $\Gamma_c$, $\Gamma_m$, $\Gamma_{cm}$ have Lipschitz continuous boundaries.
On the conduit/matrix interface $\Gamma_{cm}$, we denote by $\{\btau_i\}$ $(i=1,...,d-1)$ a local orthonormal
basis for the tangent plane to $\Gamma_{cm}$.
\begin{figure}[ht]
  \begin{center}
    \includegraphics[width=0.8\textwidth]{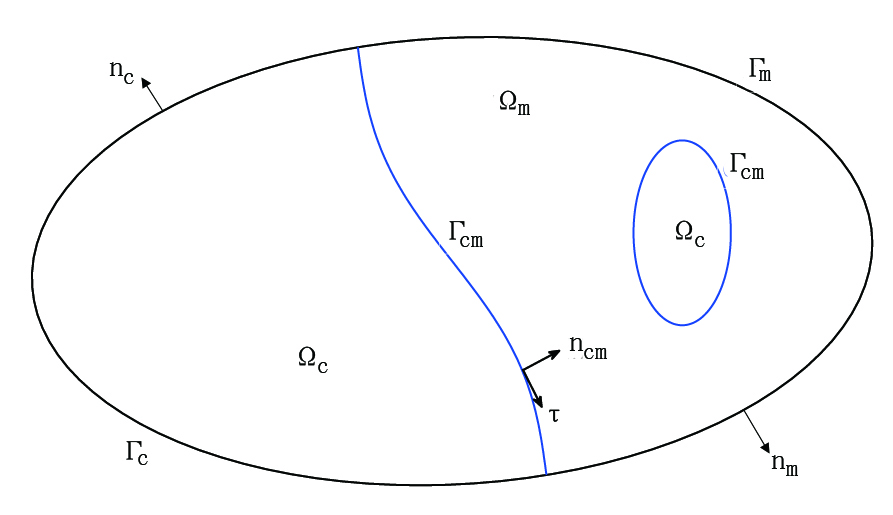}
  \end{center}
  \caption{Schematic illustration of the domain in 2D}
  \label{domain}
\end{figure}
\\
\noindent In the sequel, the subscript $m$ (or $c$) emphasizes that the variables are for the matrix part (or the conduit part). We denote by $\ub$ the mean velocity of the fluid mixture and $\vp$ the phase function related to the concentration of the fluid (volume fraction). The following convention will be assumed throughout the paper
$$
\ub|_{\Omega_m}=\ub_m, \ \ \ub|_{\Omega_c}=\ub_c,\ \ \varphi|_{\Omega_m}=\varphi_m, \ \ \varphi|_{\Omega_c}=\varphi_c.
$$

\textbf{Governing PDE system}. To the best of our knowledge, the first diffuse-interface model for incompressible two-phase flows in karstic geometry with matched densities was recently derived in \cite{HSW13} by utilizing Onsager's extremal principle (see references therein). Our aim in this paper is to study its well-posedness. Indeed, we can perform the analysis for a more general system, in which the Stokes equation can also be time-dependent. Thus, we shall consider the following Cahn-Hilliard-Stokes-Darcy system (CHSD for brevity) coupled through a set of interface boundary conditions (see \eqref{IBCi1}--\eqref{IBCi7} below):
\begin{eqnarray}
&&\rho_0\varpi\partial_t \ub_c=\nabla\cdot\mathbb{T}(\ub_c,P_c)+\mu_c\nabla\varphi_c,\ \ \mbox{in}\ \Omega_c,\label{HSCH-NSCH1}\\
&&\nabla\cdot\ub_c=0,\ \ \mbox{in}\ \Omega_c,\label{HSCH-NSCH2}\\
&&\partial_t\varphi_c+\ub_c\cdot\nabla\varphi_c={\rm div}({\rm M}(\vp_c)\nabla\mu_c),\ \ \mbox{in}\ \Omega_c,\label{HSCH-NSCH3}\\
&&\ub_m=-\frac{\rho_0 g\Pi}{\nu(\varphi_m)}\left(\nabla P_m-\mu_m\nabla\varphi_m\right),\ \ \mbox{in}\ \Omega_m,\label{HSCH-NSCH4}\\
&&\nabla\cdot\ub_m=0,\ \ \mbox{in}\ \Omega_m, \label{HSCH-NSCH5}\\
&&\partial_t\varphi_m+\ub_m\cdot\nabla\varphi_m={\rm div}({\rm M}(\vp_m)\nabla\mu_m),\ \ \mbox{in}\ \Omega_m, \label{HSCH-NSCH6}
\end{eqnarray}
where the chemical potentials $\mu_c, \mu_m$ are given by
\be
 \mu_j=\gamma \Big(-\epsilon\Delta\varphi_j+\frac{1}{\epsilon}(\vp_j^3-\vp_j)\Big), \quad \quad j\in\{c, m\} \label{chemp}.
 \ee
Here, the parameter $\varpi$ in \eqref{HSCH-NSCH1} is a nonnegative constant. When $\varpi=0$, the system \eqref{HSCH-NSCH1}--\eqref{HSCH-NSCH6} reduces to the CHSD system derived in \cite{HSW13}. $\rho_0$ represents the fluid density, and $g$ is the gravitational constant. The parameter $\gamma>0$ is related to the surface tension.  
We remark that the Stokes equation \eqref{HSCH-NSCH1} can be viewed as low Reynolds number approximation of the Navier-Stokes equation, while the Darcy equation \eqref{HSCH-NSCH4} can be viewed as the quasi-static approximation for the saturated flow model under
the assumption that the porous media pressure adjusts instantly to changes in the fluid velocity \cite{CR08,BDQ10}.

In the diffuse-interface model of immiscible two phase flows, the chemical potential $\mu$ (see Eq. \eqref{chemp}) is given by the
variational derivative of the following free energy functional
 \be
 E(\vp):=\gamma \int_\Omega \left(\frac{\epsilon}{2}|\nabla \vp|^2 +\frac{1}{\epsilon} F(\vp)\right)dx, \label{E}
 \ee
 where $F(\vp)$ is the Helmholtz free energy and usually taken to be a non-convex function of $\vp$ for immiscible two phase flows, e.g., a double-well polynomial of Ginzburg-Landau type in our present case:
 \be
 F(\vp)=\frac1{4}(\vp^2-1)^2. \label{fen}
 \ee
Singular potential of Flory-Huggins type can be treated as well, see for instance \cite{Boyer1999}.
The first term (i.e., the gradient part) of $E$ is a
diffusion term that represents the hydrophilic part of the free-energy, while the second term (i.e., the bulk part) expresses the hydrophobic part of the free-energy.
The small constant $\epsilon$  in \eqref{E} is the capillary width of the binary mixture. As the constant $\epsilon \to 0$, $\vp$ will approach $1$ and $-1$ almost everywhere, and the contribution due to the induced stress will converge to a measure-valued force term supported only on the interface between regions $\{\vp = 1\}$ and
$\{\vp = -1\}$ (cf. \cite{LS03, AbLe2012}).   The nonlinear terms $\mu_c\nabla \vp_c$ and $\mu_m\nabla \vp_m$ in the convective Cahn-Hilliard equations \eqref{HSCH-NSCH3} and \eqref{HSCH-NSCH6} can be interpreted as the ``elastic" force (or Korteweg force) exerted by the
diffusive interface of the two phase flow. This ``elastic" force converges to the surface tension at sharp interface limit $\epsilon \to 0$ at least heuristically (cf. e.g., \cite{LS03,LT98}). Since the value of $\gamma$ does not affect the analysis, we simply set $\gamma=1$ throughout the rest of the paper.
Likewise, we set the fluid density $\rho_0$ and gravitational constant $g$ to be $1$ without loss of generality.

The two phase flow in the conduit part and matrix part is described by the Stokes equation \eqref{HSCH-NSCH1} and the Darcy equation \eqref{HSCH-NSCH4}, respectively. In  \eqref{HSCH-NSCH1}, the Cauchy stress tensor $\mathbb{T}$ is given by
$$\mathbb{T}(\ub_c,p_c) = 2\nu(\vp_c)\mathbb{D}(\ub_c)-P_c\mathbb{I}$$
where $\mathbb{D}(\ub_c)=\frac{1}{2}(\nabla \ub_c + \nabla^T \ub_c)$ is the symmetric rate of deformation tensor and $\mathbb{I}$ is the $d\times d$ identity matrix. Besides, $P_c$ and $P_m$ stand for the modified pressures that also absorb the effects due to gravitation. The viscosity and the mobility of the CHSD model are denoted by $\nu$ and $\mathrm{M}$, respectively. They are assumed to be suitable functions that may depend on the phase function $\vp$ (see Section \ref{WM}). 
$M(\varphi)$  is taken to be the same (function of the phase function) for the conduit and the matrix for simplicity.
 In Eq. \eqref{HSCH-NSCH4}, $\Pi$ is a $d\times d$ matrix standing for permeability of the porous media. It is related to the hydraulic conductivity tensor of the porous medium $\mathbb{K}$ through the relation $\Pi=\frac{\nu\mathbb{K}}{\rho_0 g}$. In the literature,  $\mathbb{K}$ is usually assumed to be a bounded, symmetric and uniformly positive definite matrix but could be heterogeneous \cite{Bear}.

Next, we describe the initial  boundary (or interface) conditions of the CHSD system \eqref{HSCH-NSCH1}--\eqref{HSCH-NSCH6}.

\textbf{Initial conditions}. The CHSD System \eqref{HSCH-NSCH1}--\eqref{HSCH-NSCH6} is subject to the initial conditions
\bea
&& \ub_c|_{t=0}=\ub_0(x),  \quad \text{in} \Omega_c,\label{ini1}\\
&& \vp|_{t=0}=\vp_{0}(x), \quad \text{in} \Omega. \label{ini2}
\eea
 In particular, when $\varpi=0$, we do not need the initial condition \eqref{ini1} for $\ub_c$. \medskip

\textbf{Boundary conditions on $\Gamma_c$ and $\Gamma_m$}. The boundary conditions on $\Gamma_c$ and $\Gamma_m$ take the following form:
\begin{eqnarray}
&& \ub_c=\mathbf{0},\qquad \text{on}\ \Gamma_c,\label{IBC0}\\
&& \ub_m\cdot\mathbf{n}_m=0,\quad \text{on}\ \Gamma_m,\label{IBC1}\\
&&\frac{\partial\varphi_c}{\partial \mathbf{n}_c}=\frac{\partial\mu_c}{\partial \mathbf{n}_c}=0,\quad \text{on}\ \Gamma_c,\label{IBC2}\\
&&\frac{\partial\varphi_m}{\partial \mathbf{n}_m}=\frac{\partial\mu_m}{\partial \mathbf{n}_m}=0, \quad \text{on}\ \Gamma_m.\label{IBC3}
\end{eqnarray}

\textbf{Interface conditions on $\Gamma_{cm}$}. The CHSD system \eqref{HSCH-NSCH1}--\eqref{HSCH-NSCH6} are coupled through the following set of interface conditions:
\begin{eqnarray}
 &&\varphi_m =\varphi_c,\quad \mbox{on}\ \Gamma_{cm},\label{IBCi1}\\
&&\mu_m=\mu_c,\quad \mbox{on}\ \Gamma_{cm},\label{IBCi3}\\
&& \frac{\partial \vp_m}{\partial \mathbf{n}_{cm}}=\frac{\partial \vp_c}{\partial \mathbf{n}_{cm}},\ \ \mbox{on}\ \Gamma_{cm}\\
&& \frac{\partial \mu_m}{\partial \mathbf{n}_{cm}}=\frac{\partial \mu_c}{\partial \mathbf{n}_{cm}},\ \ \mbox{on}\ \Gamma_{cm}\label{IBCi4}\\
&&\ub_m\cdot\mathbf{n}_{cm}=\ub_c\cdot\mathbf{n}_{cm},\quad \mbox{on}\ \Gamma_{cm},\label{IBCi5}\\
&&-\mathbf{n}_{cm}\cdot(\mathbb{T}(\ub_c,P_c){\mathbf{n}_{cm}})=P_m,\ \ \mbox{on}\ \Gamma_{cm},\label{IBCi6}\\
&& -\btau_i\cdot(\mathbb{T}(\ub_c,P_c){\mathbf{n}_{cm}})=\alpha_{BJSJ}\frac{\nu(\vp_m)}{\sqrt{\text{trace}(\Pi)}}\btau_i\cdot\ub_c,\ \ \mbox{on}\ \Gamma_{cm},\label{IBCi7}
\end{eqnarray}
for $i=1,..,d-1$.

The first four interface conditions \eqref{IBCi1}--\eqref{IBCi4} are simply the continuity conditions for the phase function, the chemical potential and their normal derivatives, respectively. Condition \eqref{IBCi5} indicates the continuity in normal velocity that guarantees the conservation of mass, i.e., the exchange of fluid
between the two sub-domains is conservative. Condition \eqref{IBCi6}
represents the balance of two driving forces, the pressure
in the matrix and the normal component of the normal stress of the free flow in the conduit, in the normal direction along the interface. The last interface condition \eqref{IBCi7} is
the so-called
Beavers-Joseph-Saffman-Jones (BJSJ) condition (cf. \cite{Jones, Saffman}), where $\alpha_{BJSJ}\geq 0$ is an empirical constant determined by the geometry and the porous material. The BJSJ condition is a simplified variant of the well-known Beavers-Joseph (BJ) condition (cf. \cite{B-J})
that addresses the important
issue of how the porous media affects the conduit flow at the
interface:
$$-\btau_i\cdot(2\nu\mathbb
D(\mathbf{u}_c))\mathbf{n}_{cm}=\alpha_{BJ}\frac{\nu}{\sqrt{\text{trace}(\Pi)}}\btau_i\cdot(\ub_c-\ub_m), \ \ \mbox{on}\ \Gamma_{cm}, \ i=1,...,d-1.
$$
 This empirical condition essentially claims that the
tangential component of the normal stress that the free flow incurs
along the interface is proportional to the jump in the tangential
velocity over the interface. To get the BJSJ condition, the term $-\btau_i\cdot \ub_m$ on the right hand
side is simply dropped from the corresponding BJ condition.
Mathematically rigorous justification of this simplification under appropriate assumptions can be found in \cite{J-M}.

There is an abundant literature on mathematical studies of single component flows in karstic geometry \cite{CR08,DMQ2002, DiQu2003, DiQu2009, LSY2002, CGHW10, CGW10, CGHW11, ChGHW11, CGSW2013,  CR09, CR12, CGR13}.
Those aforementioned mathematical works on the flows in karst aquifers treat the case of confined saturated aquifer where only one type of fluid (e.g., water) occupies the whole region exclusively. The mathematical analysis is already a challenge due to the complicated coupling of the flows in the conduits and the surrounding matrix, which  are governed by different physical processes, the complex geometry of the network of conduits  as well as the strong heterogeneity.

The current work contributes to, to the authors' best knowledge, a first rigorous mathematical analysis of the diffuse-interface model for two phase incompressible flows in the karstic geometry. In particular, we prove the existence of global finite energy solutions in the sense of Definition \ref{defweak} to the CHSD system \eqref{HSCH-NSCH1}--\eqref{IBCi7} (see Theorem \ref{thmEx}). The proof is based on a novel semi-implicit discretization in time numerical scheme \eqref{app1}--\eqref{ConSp} that satisfies a discrete version of the dissipative energy law \eqref{EnergyLaw} (see Proposition \ref{DEE} below). One can thus establish the existence of weak solutions to the resulting nonlinear elliptic system via the Leray-Schauder degree theory (cf. \cite{De,Ab09}). Then the existence of global finite energy  solutions to the original CHSD system follows from a suitable compactness argument. We point out that our numerical scheme \eqref{app1}--\eqref{ConSp} differs from the one proposed and studied by Feng and Wise \cite{FeWi2012} (for the Cahn-Hilliard-Darcy system in simple domain) in the sense that, among others, both the elastic forcing term $\mu \nabla \varphi$ in the Stokes/Darcy equations and the convection term $\mathbf{u}\cdot \nabla \vp$ in the Cahn-Hilliard equation are treated in a fully implicit way. As a consequence, we are able to prove the existence of finite energy solutions  by only imposing the initial data $\varphi_0 \in H^1(\Omega)$, whereas in \cite{FeWi2012}  the authors have to assume $\varphi_0 \in H^2(\Omega)$ (or at least $H^1(\Omega) \cap L^\infty(\Omega)$), which is not natural in view of the basic energy law \eqref{EnergyLaw}. On the other hand, this choice of discretization brings extra difficulties such that neither the variational approach in \cite{Wise10, FeWi2012} nor the monotonicity method devised in \cite{HaWa2014} can be applied. Besides, the complexity of the domain geometry also motivates us to introduce an equivalent norm for the velocity field (Eq. \eqref{equinorm}), which is necessary for the analysis in the case of stationary Stokes equation ($\varpi=0$). After the existence result is obtained, a weak-strong uniqueness property of the weak solutions is shown via the energy method (cf. Theorem \ref{thmuni} for the precise statement). We note that existence and uniqueness of strong solutions to the coupled CHSD system \eqref{HSCH-NSCH1}--\eqref{IBCi7} is beyond the scope of this manuscript and will be addressed in a forthcoming work.

It is worth mentioning that there are a lot of works on diffuse-interface models for immiscible two phase incompressible flow with matched densities in a single domain setting. For instance, concerning the Cahn-Hilliard-Navier-Stokes system (Model H), existence of weak solutions, existence and uniqueness of strong solutions and long time dynamics are established in \cite{Boyer1999, Abels2009, ZWH2009, GaGr2010} and references therein. As for the Cahn-Hilliard-Darcy (also referred to as Cahn-Hilliard-Hele-Shaw) system in porous media or in the Hele-Shaw cell, the readers are referred to \cite{FeWi2012, WaZh2013,  WaWu2012, LTZ2013, DFW2014} for latest results.

The rest of this paper is organized as follows. In Section 2, we first introduce the appropriate functional spaces and derive a dissipative energy law associated with the CHSD system \eqref{HSCH-NSCH1}--\eqref{IBCi7}. After that, we present the definition of suitable weak solutions and state the main results of this paper. Section 3 is devoted to the existence of global finite energy weak solutions. We first obtain the existence of weak solutions to an implicit time-discretized system by the Leray-Schauder degree theory. Then the existence of finite energy weak solutions to the original CHSD system follows from a compactness argument. Finally, in Section 4 we prove the weak-strong uniqueness property of the weak solutions.

\section{Preliminaries and Main Results}
\setcounter{equation}{0}

\subsection{Functional spaces}
 We first introduce some notations.
If $X$ is a Banach space and $X'$ is its dual, then
$\langle u, v\rangle \equiv \langle u, v\rangle_{X'
,X} $ for $u \in X'$, $v\in X$ denotes the duality product. The inner product on a Hilbert space $H$ is denoted by $(\cdot, \cdot)_H$.
Let $\Omega\subset \mathbb{R}^d$ be a bounded domain, then $L^q(\Omega)$, $1 \leq q \leq\infty$ denotes the usual
Lebesgue space and $\|\cdot\|_{L^q(\Omega)}$ denotes its norm. Similarly, $W^{m,q}(\Omega)$, $m \in \mathbb{N}$, $1 \leq q \leq \infty$, denotes the usual
Sobolev space with norm $\|\cdot \|_{W^{m,p}(\Omega)}$. When $q=2$, we simply denote $W^{m,2}(\Omega)$ by $H^m(\Omega)$. Besides, the fractional order Sobolev spaces $H^s(\Omega)$ ($s\in \mathbb{R}$) are defined as in \cite[Section 4.2.1]{Tri}.  If $I$ is an interval of $\mathbb{R}^+$ and $X$ a Banach space, we use the function space $L^p(I;X)$, $1 \leq p \leq +\infty$, which consists of $p$-integrable
functions with values in $X$. Moreover, $C_w(I;X)$ denotes the topological vector space of all bounded and weakly continuous functions from $I$ to $X$, while $W^{1,p}(I,X)$ $(1\leq q<+\infty)$ stands for the space of all functions $u$ such that $u, \frac{du}{dt}\in L^p(I;X)$, where $\frac{du}{dt}$ denotes the vector valued distributional derivative of $u$. Bold characters are used to denote vector spaces.

Given $v \in L^1(\Omega)$, we denote by $\overline{v} = |\Omega|^{-1}\int_\Omega v(x)dx$ its mean value. Then we define the space
$\dot L^2(\Omega) := \{ v \in L^2(\Omega): \overline{v}=0\}$ and $\dot v=\mathrm{P}_0 v := v - \overline{v}$ the orthogonal projection onto $\dot L^2(\Omega)$. Furthermore, we denote $\dot H^1(\Omega)=H^1(\Omega)\cap \dot L^2(\Omega)$, which is a Hilbert space with inner product $(u, v)_{\dot H^1}=\int_\Omega \nabla u\cdot \nabla v dx$ due to the classical Poincar\'e inequality for functions with zero mean. Its dual space is simply denoted by $\dot{H}^{-1}(\Omega)$.

For our CHSD problem with domain decomposition, we introduce the following spaces
\bea
 \BH({\rm div}; \Omega_j) &:=&\{\mathbf{w}\in \mathbf{L}^2(\Omega_j)~|~\nabla \cdot \mathbf{w}\in \mathbf{L}^2(\Omega_j)\}, \quad \quad j\in \{c,m\},\non\\
\mathbf{H}_{c,0}&:=&\{\mathbf{w}\in \BH^1(\Omega_c)~|~\mathbf{w}=\mathbf{0}\text{ on
}\Gamma_{c}\},
 \non\\
\mathbf{H}_{c,\text{div}}&:=&\{
\mathbf{w}\in\mathbf{H}_{c,0}~|~\nabla \cdot\mathbf{w}=0\}, \nonumber
 \\
\mathbf{H}_{m,0}&:=&\{\mathbf{w}\in \BH({\rm div}; \Omega_m)~|~\mathbf{w}\cdot \mathbf{n}_m=0 \ \text{on}\ \Gamma_{m}\}, \non\\
\mathbf{H}_{m,\mathrm{div}}&:=&\{\mathbf{w}\in\mathbf{H}_{m,0}~|~\nabla \cdot\mathbf{w}=0\},\non\\
 X_m&:=& \dot{H}^1(\Omega_m).\non
 \eea
We denote $(\cdot, \cdot)_c$, $(\cdot, \cdot)_m$ the inner products on the spaces $L^2(\Omega_c)$, $L^2(\Omega_m)$, respectively (also for the corresponding vector spaces). The inner product on $L^2(\Omega)$ is simply denoted by $(\cdot, \cdot)$. Then it is clear that
$$
(u,v)=(u_m,v_m)_m+(u_c,v_c)_c, \quad \|u\|_{L^2(\Omega)}^2=\|u_m\|_{L^2(\Omega_m)}^2+\|u_c\|_{L^2(\Omega_c)}^2,
$$
where $u_m:=u|_{\Omega_m}$ and $u_c:=u|_{\Omega_c}$.

On the interface $\Gamma_{cm}$, we consider the fractional Sobolev spaces
$H^{\frac12}_{00}(\Gamma_{cm})$ and $H^\frac12(\Gamma_{cm})$ for (Lipschitz) surface $\Gamma_{cm}$ when $d=3$ or curve when $d=2$ with the following equivalent norms (see \cite[pp.-66]{L-M}, or \cite{Grisvard1985}):
\begin{eqnarray}
&& \|u\|_{H^\frac12(\Gamma_{cm})}^2=\int_{\Gamma_{cm}} |u|^2 dS +\int_{\Gamma_{cm}}\int_{\Gamma_{cm}}\frac{|u(x)-u(y)|^2}{|x-y|^d} dxdy,\nonumber\\
&& \|u\|_{H_{00}^\frac12(\Gamma_{cm})}^2=\|u\|_{H^\frac12(\Gamma_{cm})}^2
+\int_{\Gamma_{cm}} \frac{|u(x)|^2}{\rho(x, \partial \Gamma_{cm})} dx,\nonumber
\end{eqnarray}
where $\rho(x, \partial \Gamma_{cm})$ denotes the distance from $x$ to $\partial \Gamma_{cm}$. The above norms are not equivalent except when $\Gamma_{cm}$  is a closed surface or curve (cf. \cite{CGR13}). If $\Gamma_{cm}$ is a subset of $\partial \Omega_c$ with positive measure, then $H^{\frac12}_{00}(\Gamma_{cm})$ is a trace space of functions of $H^1(\Omega_c)$ that vanish on $\partial \Omega_c\setminus \Gamma_{cm}$. Similarly in the vectorial case, we have $\BH^{\frac12}_{00}(\Gamma_{cm})=\BH_{c,0}|_{\Gamma_{cm}}$.  $H^{\frac12}_{00}(\Gamma_{cm})$ is a non-closed subspace of
$H^{\frac12}(\Gamma_{cm})$ and has a continuous zero extension to
$H^{\frac12}(\partial\Omega_{c})$.
For $H^{\frac12}_{00}(\Gamma_{cm})$, we have the following continuous embedding result (cf. \cite{CGHW10}): $H^{\frac12}_{00}(\Gamma_{cm})\subsetneqq H^{\frac12}(\Gamma_{cm}) \subsetneqq H^{-\frac12}(\Gamma_{cm}) \subsetneqq(H^{\frac12}_{00}(\Gamma_{cm}))'$.
We note that
$
H^{-\frac12}(\partial\Omega_c)|_{\Gamma_{cm}}\nsubseteq H^{-\frac12}(\Gamma_{cm})$ and $H^{-\frac12}(\partial\Omega_c)|_{\Gamma_{cm}}\subset(H^{\frac12}_{00}(\Gamma_{cm}))'$, where the space
$H^{-\frac12}(\partial\Omega_c)|_{\Gamma_{cm}}$ is defined in the
following way: for all $f\in H^{-\frac12}(\partial\Omega_c)|_{\Gamma_{cm}}$ and $ g\in
H^{\frac12}_{00}(\Gamma_{cm})$,
$\langle f,g\rangle_{H^{-\frac12}(\partial\Omega_c)|_{\Gamma_{cm}},\,
H^{\frac12}_{00}(\Gamma_{cm})}:=\langle f,\widetilde g\rangle_{H^{-\frac12}(\partial\Omega_c),\,
H^{\frac12}(\partial\Omega_c)}$ with $\widetilde g$ being the zero
extension of $g$ to $\partial\Omega_c$.

For any $\mathbf{u} \in \mathbf{H}({\rm div}, \Omega_c)$, its normal component  $\mathbf{u} \cdot \mathbf{n}_{cm}$ is well defined in $(H^{\frac12}_{00}(\Gamma_{cm}))^\prime$,  and for all $q \in H^1(\Omega_c)$ such that $q =0$ on $\partial \Omega_c \backslash \Gamma_{cm}$, we have
\begin{align}
(\nabla \cdot \mathbf{u}, q)_c=(\mathbf{u}, \nabla q)_c+\langle \mathbf{u}\cdot \mathbf{n}_{cm}, q\rangle_{(H^{\frac12}_{00}(\Gamma_{cm}))^\prime,\,  H^{\frac12}_{00}(\Gamma_{cm})}.\non
\end{align}
Similar identity holds on the matrix domain $\Omega_m$.

\subsection{Basic energy law}
An important feature of the CHSD system \eqref{HSCH-NSCH1}--\eqref{IBCi7} is that it obeys a dissipative energy law. To this end, we first note that the total energy of the coupled system is given by:
\be
\mathcal{E}(t)=\int_{\Omega_c}\frac{\varpi}{2}|\ub_c|^2 dx+\int_{\Omega}\left[\frac{\epsilon}{2}|\nabla\varphi|^2+\frac{1}{\epsilon}F(\varphi)\right]dx.\label{totenergy}
\ee
Then we have the following formal result:
\begin{lemma}[Basic energy law]\label{BEL}
Let $(\ub_m, \ub_c, \vp)$ be a smooth solution to the initial boundary value problem \eqref{HSCH-NSCH1}--\eqref{IBCi7}. Then $(\ub_m,  \ub_c, \vp)$ satisfies the following basic energy law:
\be
\frac{d}{dt} \mathcal{E}(t)=-\mathcal{D}(t) \le 0,\quad \forall\, t\geq 0, \label{EnergyLaw}
\ee
 where the energy dissipation $\mathcal{D}$ is given by
\bea
\mathcal{D}(t)&=& \int_{\Omega_m} \nu(\varphi_m)\Pi^{-1}|\ub_m|^2dx
+\int_{\Omega_c}2\nu(\vp_c)|\mathbb{D}(\ub_c)|^2dx
\non\\
&& +\int_{\Omega}{\rm M}(\vp)|\nabla\mu(\vp)|^2dx
+\frac{\alpha_{BJSJ}}{\sqrt{{\rm trace}(\Pi)}}\sum_{i=1}^{d-1}\int_{\Gamma_{cm}} \nu(\vp)|\ub_c\cdot\btau_i|^2 dS.\label{D}
\eea
\end{lemma}
\begin{proof}
For the conduit part, multiplying the equations \eqref{HSCH-NSCH1}, \eqref{HSCH-NSCH3} by $\ub_c$ and  $\mu(\varphi_c)$, respectively,  integrating over $\Omega_c$, and adding the resultants together, we get
\bea
&&\frac{d}{dt}\int_{\Omega_c}\frac{\varpi}{2}|\ub_c|^2 dx +\int_{\Omega_c} \partial_t\varphi_c\mu(\vp_c)dx \non\\
& =& \int_{\Omega_c} [\nabla\cdot\mathbb{T}(\ub_c,P_c)]\cdot\ub_c dx+\int_{\Omega_c}\mu(\vp_c){\rm div}({\rm M}(\vp_c)\nabla\mu(\vp_c))dx.\non
\eea
After integration by parts and using the boundary conditions, we obtain that
\bea
&&\frac{d}{dt}\int_{\Omega_c}\left[\frac{\varpi}{2}|\ub_c|^2 + \frac{\epsilon}{2}|\nabla\varphi_c|^2+\frac{1}{\epsilon}F(\varphi_c)\right]dx +\int_{\Omega_c}{\rm M}(\vp_c)|\nabla\mu(\vp_c)|^2dx
\non\\
&=&\int_{\Omega_c}[\nabla\cdot\mathbb{T}(\ub_c,P_c)]\cdot\ub_c dx
+\int_{\Gamma_{cm}}{\rm M}(\vp_c)\mu(\vp_c)\frac{\partial\mu(\vp_c)}{\partial \mathbf{n}_{cm}}dS\non\\
&& +\epsilon \int_{\Gamma_{cm}}\partial_t\varphi_c\frac{\partial\varphi_c}{\partial\mathbf{n}_{cm}}dS.
 \label{HSCH-NSCH:NSpart}
\eea
Applying the divergence theorem to the first term on the right-hand side of \eqref{HSCH-NSCH:NSpart}, we infer from the boundary conditions \eqref{IBC0}, \eqref{IBCi6}, \eqref{IBCi7} and the incompressibility condition \eqref{HSCH-NSCH2} that
\begin{eqnarray}
&&
  \int_{\Omega_c}[\nabla\cdot\mathbb{T}(\ub_c,P_c)]\cdot\ub_c dx\non\\
&=&
  \int_{\Gamma_{cm}}(\mathbb{T}(\ub_c,P_c)\mathbf{n}_{cm})\cdot\ub_c dS-\int_{\Omega_c}\mathbb{T}(\ub_c,P_c):\nabla\ub_c dx\non\\
&=&
  \sum_{i=1}^2\int_{\Gamma_{cm}}(\btau_i^T\mathbb{T}(\ub_c,P_c)\mathbf{n}_{cm})(\ub_c\cdot\btau_i) dS\non\\
&&
  +\int_{\Gamma_{cm}}(\mathbf{n}_{cm}^T\mathbb{T}(\ub_c,P_c)\mathbf{n}_{cm})(\ub_c\cdot\mathbf{n}_{cm}) dS\non\\
&&
  -\int_{\Omega_c}(2\nu(\vp_c)\mathbb{D}(\ub_c)-P_c \mathbb{I}):\nabla\ub_c dx\non\\
&=&
  -\frac{\alpha_{BJSJ}}{\sqrt{{\rm trace}(\Pi)}}\sum_{i=1}^{d-1}\int_{\Gamma_{cm}} \nu(\vp_m)|\ub_c\cdot\btau_i|^2 dS-\int_{\Gamma_{cm}}P_m(\ub_c\cdot\mathbf{n}_{cm}) dS\non\\
&&
  -\int_{\Omega_c}2\nu(\vp_c)|\mathbb{D}(\ub_c)|^2dx.
\label{HSCH-NSCH:NSpart:tensor}
\end{eqnarray}

Next, we consider the matrix part. Multiplying the equation \eqref{HSCH-NSCH6} by $\mu(\varphi_m)$ and integrating over $\Omega_m$, we get
\begin{eqnarray}
  \int_{\Omega_m}\partial_t\varphi_m\mu(\varphi_m)+(\ub_m\cdot\nabla\varphi_m)\mu(\varphi_m)dx= \int_{\Omega_m}\mu(\varphi_m){\rm div}({\rm M}(\vp_m)\nabla \mu(\varphi_m)) dx.
\end{eqnarray}
On the other hand, we infer from the Darcy equation \eqref{HSCH-NSCH1} that
 $$
 \mu(\varphi_m)\nabla \varphi_m=\nu(\varphi_m)\Pi^{-1}\ub_m+\nabla P_m.
 $$
Using this fact and integration by parts, we infer from the boundary condition \eqref{IBC3} that
\begin{eqnarray}
 \label{HSCH-NSCH:HSpart}
&&
 \frac{d}{dt}\int_{\Omega_m}\left[\frac{\epsilon}{2}|\nabla\varphi_m|^2+\frac{1}{\epsilon}F(\varphi_m)\right]dx \non\\
&&
 + \epsilon \int_{\Gamma_{cm}}\partial_t\varphi_m\frac{\partial\varphi_m}{\partial\mathbf{n}_{cm}}dS + \int_{\Omega_m}\left( \nu(\varphi_m)\Pi^{-1}|\ub_m|^2+ \ub_m\cdot \nabla P_m\right)dx\non\\
&=&
 -\int_{\Gamma_{cm}}{\rm M}(\vp_m)\mu(\vp_m)\frac{\partial\mu(\vp_m)}{\partial \mathbf{n}_{cm}}dS-\int_{\Omega_m}{\rm M}(\vp_m)|\nabla\mu(\vp_m)|^2dx,
\end{eqnarray}
where we recall that $\mathbf{n}_{cm}$ denotes the
unit normal to interface ${\Gamma_{cm}}$ pointing from the conduit to the
matrix.
By the divergence theorem and the incompressibility condition \eqref{HSCH-NSCH5}, we get
\begin{eqnarray}
\int_{\Omega_m}\ub_m\cdot\nabla P_m dx
 &= &
  \int_{\Omega_m}\left[\nabla\cdot(P_m \ub_m)-P_m(\nabla\cdot\ub_m) \right]dx\non\\
&=&
  -\int_{{\Gamma_{cm}}}P_m\ub_m\cdot\mathbf{n}_{cm}dS
\end{eqnarray}
Then \eqref{HSCH-NSCH:HSpart} becomes
\begin{eqnarray}
 \label{HSCH-NSCH:HSpart:Final}
&&
  \frac{d}{dt}\int_{\Omega_m}\left[\frac{\epsilon}{2}|\nabla\varphi_m|^2+\frac{1}{\epsilon}F(\varphi_m)\right]dx\non\\
&&
   + \int_{\Omega_m} \left( \nu(\varphi_m)\Pi^{-1}|\ub_m|^2+{\rm M}(\vp_m)|\nabla\mu(\vp_m)|^2\right)dx\non\\
&=&
  -\int_{\Gamma_{cm}}{\rm M}(\vp_m)\mu(\vp_m)\frac{\partial\mu(\vp_m)}{\partial \mathbf{n}_{cm}}dS-\epsilon \int_{\Gamma_{cm}}\partial_t\varphi_m\frac{\partial\varphi_m}{\partial\mathbf{n}_{cm}}dS\non\\
  && +\int_{\Gamma_{cm}}P_m\ub_m\cdot\mathbf{n}_{cm}dS.
\end{eqnarray}

Finally, combining \eqref{HSCH-NSCH:NSpart}, \eqref{HSCH-NSCH:NSpart:tensor} and \eqref{HSCH-NSCH:HSpart:Final}, using the definition of $\varphi$ as well as the continuity conditions \eqref{IBCi1}--\eqref{IBCi3} on interface $\Gamma_{cm}$, we can cancel the boundary terms and conclude the basic energy law \eqref{EnergyLaw}. The proof is complete.
\end{proof}


\subsection{Weak formulation and main results}\label{WM}
We make the following assumptions on viscosity $\nu$,   mobility coefficient $\mathrm{M}$ as well as the permeability matrix $\Pi$:
 \begin{itemize}
 \item[(A1)] $\nu\in C^1(\mathbb{R})$,  $\underline{\nu} \leq \nu(s)\leq\bar{\nu}  $ and $|\nu'(s)|\leq \tilde{\nu}$ for $s\in \mathbb{R}$, where $\bar{\nu}$, $\underline{\nu}$ and $\tilde{\nu}$ are positive constants.
 \item[(A2)] ${\rm M}\in C^1  (\mathbb{R})$, $\underline{m} \leq {\rm M}(s)\leq \bar m $ and $|{\rm M}'(s)|\leq \tilde{m}$ for $s\in \mathbb{R}$,  where $\bar{m}$, $\underline{m}$ and $\tilde{m}$  are positive constants.
 \item[(A3)]  The permeability $\Pi$ is isotropic, bounded from above and below (so is the hydraulic conductivity tensor $\mathbb{K}$),  namely, $\Pi=\kappa(x)\mathbb{I}$ with $\mathbb{I}$ being the $d\times d$ identity matrix and $\kappa(x)\in L^\infty (\Omega)$ such that there exist  $\bar{\kappa}>\underline{\kappa}>0$, $\underline{\kappa}\leq \kappa(x)\leq \bar{\kappa}$ a.e. in $\Omega$.
 \end{itemize}

Below we introduce the notion of finite energy weak solution to the CHSD system \eqref{HSCH-NSCH1}--\eqref{IBCi7} as well as its corresponding weak formulation.

\begin{definition}\label{defweak} Suppose that $d=2,3$ and $T>0$ is arbitrary. Let $\alpha=\frac85$ when $d=3$ and $\alpha <2 $ being arbitrary close to $2$ when $d=2$.

\textbf{Case 1}: $\varpi>0$. We consider the initial data $\ub_0(x)\in \mathbf{L}^2(\Omega_c)$, $\vp_{0}\in H^1(\Omega)$. The functions $(\ub_c, \mathbf{u}_m, P_m, \vp, \mu)$ with the following properties
\bea
&& \ub_c\in L^\infty(0,T; \mathbf{L}^2(\Omega_c))\cap L^2(0, T; \mathbf{H}_{c,\mathrm{div}})\cap W^{1,\alpha}(0,T; (\mathbf{H}^1(\Omega_c))'),\label{regubc}\\
&& \ub_m\in L^2(0, T; \mathbf{L}^2(\Omega_m)),\label{regubm}\\
&& P_m \in L^\alpha(0,T; X_m),\label{regpm}\\
&& \vp\in L^\infty(0,T; H^1(\Omega))\cap L^2(0,T; H^3(\Omega))\cap W^{1,\alpha}(0;T; (H^1(\Omega))'),\\
&& \mu\in L^2(0, T; H^1(\Omega)),
\eea
is called a finite energy weak solution of the CHSD system \eqref{HSCH-NSCH1}--\eqref{IBCi7}, if the following conditions are satisfied:

(1) For any $\mathbf{v}_c\in C^\infty_0((0,T); \mathbf{H}_{c,\mathrm{div}})$ and $q_m\in C([0,T]; X_m)$,
 \begin{eqnarray}
 &&-\varpi\int_0^T(\ub_c,\partial_t\mathbf{v}_c)_c dt +2\int_0^T(\nu(\vp_c)\mathbb{D}(\ub_c),\mathbb{D}(\mathbf{v}_c))_cdt
 \non\\
&& + \int_0^T \left(\frac{\Pi}{\nu(\varphi_m)}\Big[\nabla P_m- \mu(\varphi_m)\nabla\varphi_m\Big] ,\nabla q_m\right)_m dt
\non\\
&&+\sum_{i=1}^{d-1}\frac{\alpha_{BJSJ}}{\sqrt{{\rm trace} (\Pi)}}\int_0^T\int_{\Gamma_{cm}}\nu(\vp_m)(\ub_c\cdot\btau_i)(\mathbf{v}_c\cdot\btau_i)dSdt\non\\
&&+ \int_0^T \int_{\Gamma_{cm}} P_m (\mathbf{v}_c\cdot \mathbf{n}_{cm}) dSdt  - \int_0^T \int_{\Gamma_{cm}} (\ub_c\cdot \mathbf{n}_{cm})q_m dSdt \non\\
&=& \int_0^T (\mu(\varphi_c)\nabla \vp_c,  \mathbf{v}_c )_c dt,\label{weak1}
\end{eqnarray}
moreover, the velocity $\ub_m$ in the matrix part satisfies
 \be
\int_0^T (\ub_m, \mathbf{v}_m)_m dx =-\int_0^T\left(\frac{\Pi}{\nu(\varphi_m)}\Big[\nabla P_m-\mu(\varphi_m)\nabla\varphi_m\Big], \mathbf{v}_m\right)_m dt,\label{weak2}
 \ee
 for any $\mathbf{v}_m\in C([0, T]; \mathbf{L}^2(\Omega_m))$.

(2) For any $\phi\in C_0^\infty((0,T); H^1(\Omega))$,
\begin{eqnarray}
&&-\int_0^T(\varphi,\partial_t\phi)dt+\int_0^T({\rm M}(\vp) \nabla \mu(\vp),\nabla\phi)dt=-\int_0^T(\ub \cdot \nabla\varphi ,\phi)dt, \label{weak3}\\
&&\int_0^T(\mu(\varphi),\phi)dt=\int_0^T\left[\frac{1}{\epsilon}(f(\varphi),\phi)+\epsilon(\nabla\varphi,\nabla\phi)\right]dt. \label{weak4}
\end{eqnarray}

(3) $\ub_c|_{t=0}=\ub_0(x)$, $\vp|_{t=0}=\vp_{0}(x)$.

(4) The finite energy solution satisfies the energy inequality
\be
\mathcal{E}(t)  +\int_s^t\mathcal{D}(\tau) d\tau \leq \mathcal{E}(s), \label{Energyinq}
\ee
for all $t\in [s,T)$ and almost all $s\in [0,T)$ (including $s=0$), where the total energy $\mathcal{E}$ is given by \eqref{totenergy}.

\textbf{Case 2}: $\varpi=0$. In this case, we do not need the initial condition for $\ub_c$. The regularity property for $\ub_c$ (cf. \eqref{regubc}) is simply replaced by
 \be
 \ub_c\in L^2(0, T; \mathbf{H}_{c,\mathrm{div}}).\label{regubca}
 \ee
 The finite energy weak solution $(\ub_c, \mathbf{u}_m, P_m, \vp, \mu)$ still fulfills the above properties (1)--(4) with $\varpi=0$ in corresponding formulations.
  \end{definition}
 \begin{remark}
In the above weak formulation \eqref{weak1}--\eqref{weak2}, the reason we do not break the force term $\nabla P_m-\mu(\varphi_m)\nabla\varphi_m$ is that this term (or equivalently, the velocity in the matrix part $\ub_m$) has better regularity/integrability than its two components (see \eqref{regubm}--\eqref{regpm}).
\end{remark}
\begin{remark}
We note that the interface boundary conditions \eqref{IBC1}--\eqref{IBCi7} are enforced as a consequence of the weak formulation stated above. Note also that the pressure terms $P_c$ and $P_m$ are only uniquely determined up to an additive constant in the strong form \eqref{HSCH-NSCH1}--\eqref{IBCi7}, i.e., they satisfy the same set of equations with the same boundary conditions as well as interface conditions after being shifted by the same constant. As a consequence, it makes sense to seek $P_m$ in the space $\dot{H}^1(\Omega_m)$ (i.e., $X_m$).  The equivalence for smooth solutions between the weak formulation and the classical form can be verified in a straightforward way.
\end{remark}

Now we are in a position to state the mains results of this paper:

\bt[Existence of finite energy weak solutions]\label{thmEx}
 Suppose that $d=2,3$ and the assumptions (A1)--(A3) are satisfied.
\begin{itemize}
 \item[(i)] If $\varpi>0$, for any $\ub_0\in \mathbf{L}^2(\Omega_c)$, $\varphi_0 \in H^1(\Omega)$ and $T>0$ being arbitrary, the CHSD system \eqref{HSCH-NSCH1}--\eqref{IBCi7} admits at least one global finite energy weak solution $\{\mathbf{u}_c, \ub_m, P_m, \vp, \mu\}$ in the sense of Definition \ref{defweak}.

 \item[(ii)] If $\varpi=0$, for any $\varphi_0 \in H^1(\Omega)$, the CHSD system \eqref{HSCH-NSCH1}--\eqref{IBCi7} admits at least one global finite energy weak solution $\{\mathbf{u}_c, \ub_m, P_m, \vp, \mu\}$ in the sense of Definition \ref{defweak}.
 \end{itemize}
\et

\bt[Weak-strong uniqueness]\label{thmuni}
 Let $d=2,3$, $\varpi\geq 0$ and the assumptions (A1)--(A3) be satisfied. Suppose that $\{\mathbf{u}_c, \ub_m, P_m, \vp\}$ is a finite energy weak solution to the CHSD system \eqref{HSCH-NSCH1}--\eqref{IBCi7} in $(0,T)\times \Omega$ and $\{\tilde{\mathbf{u}}_c, \tilde{\ub}_m, \tilde{P}_m, \tilde{\vp}\}$ is a regular solution emanating from the same initial data with the following regularity conditions
 $$
 \tilde{\mathbf{u}}_c \in L^{\frac{8}{3}}(0,T;\mathbf{H}_{c, \mathrm{div}}),\quad  \tilde{\ub}_m\in L^\frac{8}{3}(0,T; \mathbf{H}_{m,\mathrm{div}}),\quad \tilde{\vp}\in L^\frac{8}{3}(0,T; H^3(\Omega)),
 $$
 then it holds
 $$ \mathbf{u}_c=\tilde{\mathbf{u}}_c,\quad \ub_m=\tilde{\ub}_m,\quad  P_m=\tilde{P}_m,\quad \vp=\tilde{\vp}.
 $$
\et


\section{Existence of Weak Solutions}\setcounter{equation}{0}
We shall apply a semi-discretization approach (finite difference in time, cf. \cite{Lions1969, Temam1977}) to prove the existence result Theorem \ref{thmEx}.
 First, a discrete in time, continuous in space numerical scheme is proposed and shown to be mass-conservative and energy law preserving. Then, the existence of weak solutions to the discretized system is proved by the Leray-Schauder degree theory. Last, an approximate solution is constructed and its convergence to the weak solution of the original CHSD system \eqref{HSCH-NSCH1}--\eqref{IBCi7} is established via a compactness argument.

\subsection{A time discretization scheme}
Here we propose a semi-implicit time discretization scheme to the weak formulation  \eqref{weak1}--\eqref{weak4}. Recall our convention
$$
\varphi |_{\Omega_c}=\varphi_c, \quad \varphi |_{\Omega_m}=\varphi_m, \quad \mu |_{\Omega_c}=\mu_c, \quad \mu|_{\Omega_m}=\mu_m.
$$
For arbitrary but fixed $T>0$ and positive integer $N\in \mathbb{N}$, we denote by $\delta=\Delta t=\frac{T}{N}$ the time step size. Given $(\ub_c^k, \vp_c^k, P_m^k, \vp_m^k)$, $k=0,1,2,...,N-1$, we want to determine $(\ub_c, \vp_c, P_m, \vp_m)=(\ub_c^{k+1}, \vp_c^{k+1}, P_m^{k+1}, \vp_m^{k+1})$ as a solution of the following nonlinear elliptic system
\begin{eqnarray}
&&\varpi \left(\frac{\ub_c^{k+1}-\ub_c^k}{\delta} ,\mathbf{v}_c\right)_c+2\left(\nu(\vp_c^k)\mathbb{D}(\ub_c^{k+1}),\mathbb{D}(\mathbf{v}_c)\right)_c\non\\
&&+\left(\frac{\Pi}{\nu(\varphi_m^k)}\left[\nabla P_m^{k+1}- \mu_m^{k+1}\nabla\varphi_m^{k+1}\right] ,\nabla q_m\right)_m
\non\\
&& +\sum_{i=1}^{d-1}\frac{\alpha_{BJSJ}}{\sqrt{{\rm trace} (\Pi)}} \int_{\Gamma_{cm}}\nu(\vp^k_m)(\ub_c^{k+1}\cdot\btau_i)(\mathbf{v}_c\cdot\btau_i)dS\non\\
&& +
\int_{\Gamma_{cm}} P_m^{k+1} (\mathbf{v}_c\cdot \mathbf{n}_{cm}) dS -\int_{\Gamma_{cm}} (\ub_c^{k+1}\cdot \mathbf{n}_{cm})q_m dS \non\\
&=&(\mu_c^{k+1}\nabla\varphi_c^{k+1},\mathbf{v}_c)_c,\label{app1}
 \end{eqnarray}
\begin{equation}
\left(\frac{\varphi^{k+1}-\vp^k}{\delta},\phi\right) + (\ub^{k+1}\cdot \nabla \vp^{k+1} ,\phi)=- \left({\rm M}(\vp^k)\nabla \mu^{k+1},\nabla\phi\right),
   \label{app2a}
\end{equation}
\begin{equation}
(\mu^{k+1},\phi)=\frac{1}{\epsilon}\left(\widetilde{f}(\varphi^{k+1},\varphi^k),\phi\right) +\epsilon(\nabla\varphi^{k+1},\nabla\phi), \label{app2b}
\end{equation}
for any $\mathbf{v}_c\in \mathbf{H}_{c,{\rm div}}$, $q_m \in X_m$ and $\phi\in H^1(\Omega)$.
In the above formulation, the vector $\mathbf{u}^{k+1}$ satisfies $\ub^{k+1}|_{\Omega_c}=\ub_c^{k+1}$ and $\ub^{k+1}|_{\Omega_m}=\ub_m^{k+1}$, where
 \be
\ub_m^{k+1}=-\frac{\Pi}{\nu(\varphi_m^k)}\left(\nabla P_m^{k+1}-\mu_m^{k+1}\nabla\varphi_m^{k+1}\right).\label{app3}
 \ee
The function $\widetilde{f}(\phi,\psi)$ in equation \eqref{app2b} is derived from a convex splitting approximation to the nonconvex function $F(\varphi)$ (see \eqref{fen}) and it takes the following form (cf. e.g., \cite{Eyre,Wise10})
 \be
 \widetilde{f}(\phi,\psi)=\phi^3-\psi.\label{ConSp}
 \ee

\begin{remark}
We note that equations \eqref{app1}--\eqref{app3} are strongly coupled, which demands suitable choices on discretization schemes in order to prove the existence of weak solutions (see \cite{Wise10, FeWi2012} and \cite{HaWa2014} for related diffuse-interface models). Here, the advective term in the Cahn-Hilliard equation (i.e., the second term $\mathbf{u}\cdot \nabla \vp$ in equation \eqref{app2a}) and accordingly the elastic forcing term $\mu \nabla \varphi$ in equations \eqref{app1}, \eqref{app3}  are discretized fully implicitly. Under this fully implicit discretization, it is possible to preserve a discrete energy law (see Lemma \ref{DEE}) in analogy to the continuous one \eqref{EnergyLaw}, moreover it enables us to obtain the existence of weak solutions under the natural assumption on initial data such that $\varphi_0 \in H^1(\Omega)$. In \cite{Wise10, FeWi2012}, a different semi-implicit treatment of the advective term and the elastic forcing term for the Cahn-Hilliard-Darcy system in a simple domain was proposed. The discretization therein still keeps a discrete energy law while one needs to assume $\varphi_0 \in H^2(\Omega)$ (or at least $H^1(\Omega) \cap L^\infty(\Omega)$) to obtain the existence of weak solutions.
 \end{remark}

In the following content of this subsection, we will temporarily omit the superscript $k+1$ for $\mathbf{u}_c^{k+1}$, $P_m^{k+1}$, $\mathbf{u}_m^{k+1}$, $\varphi^{k+1}$, $\mu^{k+1}$ for the sake of simplicity. Besides, we just provide the proof for the case $\varpi>0$, while the argument can be easily adapted to the simpler case $\varpi=0$ with minor modifications.

A few \emph{a priori} estimates can be readily derived.  First, one can deduce that the above numerical scheme keeps the mass conservation property.
\begin{lemma}\label{LMC}
Suppose that $\ub_c^k\in \mathbf{L}^2(\Omega_c)$, $\vp^k\in H^1(\Omega)$ and
$\{\ub_c, P_m, \varphi, \mu\}\in  \mathbf{H}_{c,{\rm div}} \times X_m\times  H^3(\Omega)\times  H^1(\Omega)$ solve the nonlinear system \eqref{app1}--\eqref{app3}. Then $\mathbf{u}_m$ (given by \eqref{app3}) satisfies
\begin{align}\label{Fum}
\mathbf{u}_m \in \mathbf{H}_{m, \mathrm{div}}, \quad \mathbf{u}_m \cdot \mathbf{n}_{cm}= \mathbf{u}_c \cdot \mathbf{n}_{cm}\in H^{\frac{1}{2}}(\Gamma_{cm}).
\end{align}
Moreover,  the following mass-conservation holds
\begin{align}\label{MC}
\int_{\Omega}\varphi\, dx=\int_{\Omega} \varphi^k \, dx.
\end{align}
\end{lemma}

\begin{proof}
It is clear from equation \eqref{app3} and the Sobolev embedding theorem ($d\leq 3$) that $\mathbf{u}_m \in \mathbf{L}^2(\Omega_m)$.
Taking the test function $\mathbf{v}_c =0$ in equation \eqref{app1} and utilizing equation \eqref{app3}, one obtains
\be
\big({\mathbf{u}}_m, \nabla q_m \big)_m  -\int_{\Gamma_{cm}} ({\mathbf{u}}_c \cdot \mathbf{n}_{cm}) q_m dS =0, \quad  \forall \, q_m \in X_m,\label{wofum}
\ee
which easily yields that $\nabla \cdot {\mathbf{u}}_m=0$ in the sense of distribution and then $\mathbf{u}_m \in \mathbf{H}(\mathrm{div}; \Omega_m)$. Thus, the normal component ${\mathbf{u}}_m \cdot \mathbf{n}$ is well-defined in $H^{-\frac{1}{2}}(\partial \Omega_m)$ ($\mathbf{n}$ denotes the unit outer  normal on $\partial \Omega_m$ and it corresponds to $\mathbf{n}_m$ on $\Gamma_m$ and to $-\mathbf{n}_{cm}$ on $\Gamma_{cm}$, respectively). Applying Green's formula to the first term in equation \eqref{wofum} gives that
$$ {\mathbf{u}}_m \cdot \mathbf{n}_m =0 \text{ in }  H^{-\frac{1}{2}}(\Gamma_m) \quad \text{and}\quad {\mathbf{u}}_m \cdot \mathbf{n}_{cm} = {\mathbf{u}}_c \cdot \mathbf{n}_{cm} \text{ in } \big(H^{\frac{1}{2}}_{00}(\Gamma_{cm})\big)^\prime.
$$
Therefore, $\mathbf{u}_m \in \mathbf{H}_{m, \mathrm{div}}$. It follows from the trace theorem that ${\mathbf{u}}_c \cdot \mathbf{n}_{cm} \in H^{\frac{1}{2}}(\Gamma_{cm})$, then one further gets
${\mathbf{u}}_m \cdot \mathbf{n}_{cm} = {\mathbf{u}}_c \cdot \mathbf{n}_{cm}$  in $H^{\frac{1}{2}}(\Gamma_{cm})$.

The mass-conservation \eqref{MC} now follows from taking the test function $\phi=1$ in equation \eqref{app2a} and performing integration by parts.
\end{proof}
The next lemma shows that the numerical scheme \eqref{app1}--\eqref{ConSp} satisfies a discrete analogue of the basic energy law \eqref{BEL}.
\begin{lemma}\label{DEE}
Suppose that $\ub_c^k\in \mathbf{L}^2(\Omega_c)$, $\vp^k\in H^1(\Omega)$ and
$\{\ub_c, P_m, \varphi, \mu\}\in  \mathbf{H}_{c,{\rm div}} \times X_m\times  H^3(\Omega)\times  H^1(\Omega)$ solve the system \eqref{app1}--\eqref{app3}. Then the following discrete energy inequality holds
\begin{eqnarray}\label{DisEnLaw}
&&
  \mathcal{E}(\ub_c,\varphi) +\delta \left( \nu(\varphi_m^k)\Pi^{-1}\ub_m, \ub_m\right)_m
  +2\delta \left(\nu(\vp_c^k)\mathbb{D}(\ub_c), \mathbb{D}(\ub_c)\right)_c
\non\\
&& +\delta\int_\Omega \mathrm{M}(\vp^k)|\nabla\mu|^2 dx+\frac{\delta \alpha_{BJSJ}}{\sqrt{{\rm trace}(\Pi)}}\sum_{i=1}^{d-1}\int_{\Gamma_{cm}} \nu(\vp^k_m)|\ub_c\cdot\btau_i|^2 dS
  \non\\
  &&+\frac{\varpi}{2}\left(\ub_c-\ub_c^k ,\ub_c-\ub_c^k\right)_c +\frac{\epsilon}{2}\|\nabla (\varphi-\varphi^k)\|_{L^2(\Omega)}^2+ \frac{1}{2\epsilon}\|\vp-\vp^k\|_{L^2(\Omega)}^2\non\\
&\le& \mathcal{E}(\ub_c^k,\varphi^k),
\end{eqnarray}
where the energy functional $\mathcal{E}$ is defined in \eqref{totenergy}.
\end{lemma}

\begin{proof}
 Taking $\mathbf{v}_c=\mathbf{u}_c$, $q_m=P_m$ in \eqref{app1}, using \eqref{app3} and the elementary identity
\begin{align}
a\cdot(a-b)=\frac{1}{2}\left(|a|^2-|b|^2+|a-b|^2\right), \quad \forall\, a, b\in \mathbb{R} \ \ \text{or}\ \ \mathbb{R}^d.\label{eek1}
\end{align}
  we have
\begin{eqnarray}
&&\frac{\varpi}{2\delta}\left(\ub_c,\mathbf{u}_c\right)_c+ \frac{\varpi}{2\delta}\left(\ub_c-\ub_c^k ,\ub_c-\ub_c^k\right)_c +2\left(\nu(\vp_c^k)\mathbb{D}(\ub_c),\mathbb{D}(\mathbf{u}_c)\right)_c\non\\
&&+\left( \nu(\varphi_m^k)\Pi^{-1}\ub_m, \ub_m\right)_m +\sum_{i=1}^{d-1}\frac{\alpha_{BJSJ}}{\sqrt{{\rm trace} (\Pi)}} \int_{\Gamma_{cm}}\nu(\vp^k_m)|\ub_c\cdot\btau_i|^2dS\non\\
&=&\frac{\varpi}{2\delta}\left(\ub_c^k,\mathbf{u}_c^k\right)_c+  (\mu\nabla\varphi,\mathbf{u}).\label{eapp1}
 \end{eqnarray}
By a direct calculation, we infer from the definition of the convex splitting function $\widetilde{f}$ that
\bea
\widetilde{f}(\phi,\psi)(\phi-\psi)
&=&
F(\phi)-F(\psi)+\frac{1}{4}(\phi^2-\psi^2)^2+\frac{1}{2}(\phi-\psi)^2+\frac{1}{2}\phi^2(\phi-\psi)^2\non
\\
&\geq & F(\phi)-F(\psi)+\frac{1}{2}(\phi-\psi)^2.\label{eek2}
\eea
Then taking the test functions $\phi=\mu$ in \eqref{app2a} and $\phi=\varphi-\varphi^k$ in \eqref{app2b},  after integration by parts, we infer from \eqref{eek1} and \eqref{eek2} that
\begin{eqnarray}
\left(\frac{\varphi-\vp^k}{\delta},\mu \right) + (\ub\cdot \nabla \vp ,\mu)+ \int_\Omega \mathrm{M}(\vp^k)|\nabla\mu|^2 dx=0,
\end{eqnarray}
where
\begin{eqnarray}
\left(\varphi-\vp^k,\mu \right)
&=&\frac{1}{\epsilon}\left(\widetilde{f}(\varphi,\varphi^k),\varphi-\vp^k\right) +\epsilon(\nabla\varphi,\nabla(\varphi-\vp^k))\non\\
&\geq & \frac{\epsilon}{2}\|\nabla \varphi\|_{L^2(\Omega)}^2+\frac{\epsilon}{2}\|\nabla (\varphi-\varphi^k)\|_{L^2(\Omega)}^2-\frac{\epsilon}{2}\|\nabla \varphi^k\|_{L^2(\Omega)}^2\non\\
   &&     +\frac{1}{\epsilon}\left(F(\vp)-F(\vp^k), 1\right) + \frac{1}{2\epsilon}\|\vp-\vp^k\|_{L^2(\Omega)}^2.\label{eapp2b}
\end{eqnarray}
Combining the above estimates \eqref{eapp1}--\eqref{eapp2b} together, we easily conclude the discrete energy inequality  \eqref{DisEnLaw}.
\end{proof}

\subsection{Existence of weak solutions to the discrete problem}

In order to prove the existence of solutions to the discrete problem \eqref{app1}--\eqref{app3}, we shall adapt an argument involving the Leray-Schauder degree theory (cf. e.g., \cite{De}) that has been used in \cite{Ab09} to show the existence of weak solutions to a diffuse-interface model in simple domain with general densities. The idea is to rewrite the system \eqref{app1}--\eqref{app2b} in terms of suitable "good" operator denoted by  $\mathcal{T}_k$ and "bad" operator denoted by $\mathcal{G}_k$ such that
\begin{align}
\mathcal{T}_k(\mathbf{w})=\mathcal{G}_k(\mathbf{w}), \label{Abs}
\end{align}
where $\mathbf{w}:=\{\ub_c, P_m, \varphi, \mu\}$ is the solution. More precisely, in the abstract equation \eqref{Abs} the operators $\mathcal{T}_k: \mathbf{X} \rightarrow \mathbf{Y}$ and $\mathcal{G}_k: \mathbf{X} \rightarrow \mathbf{Y} $ (see \eqref{Tk}--\eqref{Gk} for their detailed definition and the associated spaces $\mathbf{X}$ and $\mathbf{Y}$ will be specified in \eqref{XY}) basically correspond to, respectively, the left-hand side and right-hand side of the following reformulation of the system \eqref{app1}--\eqref{app2b} (dropping the superscript $k+1$ for simplicity as mentioned before)
\begin{eqnarray}
&&\left(\ub_c,\mathbf{v}_c\right)_c+2\left(\nu(\vp_c^k)\mathbb{D}(\ub_c),\mathbb{D}(\mathbf{v}_c)\right)_c+\left(\frac{\Pi}{\nu(\varphi_m^k)}\nabla P_m,\nabla q_m\right)_m\non\\
&& +\sum_{i=1}^{d-1}\frac{\alpha_{BJSJ}}{\sqrt{{\rm trace} (\Pi)}} \int_{\Gamma_{cm}}\nu(\vp^k_m)(\ub_c\cdot\btau_i)(\mathbf{v}_c\cdot\btau_i)dS\non\\
&& +
\int_{\Gamma_{cm}} P_m (\mathbf{v}_c\cdot \mathbf{n}_{cm}) dS -\int_{\Gamma_{cm}} (\ub_c\cdot \mathbf{n}_{cm})q_m dS \label{wapp1}\\
&=&(\mu_c\nabla\varphi_c,\mathbf{v}_c)_c+\left(\ub_c,\mathbf{v}_c\right)_c+\left(\frac{\Pi}{\nu(\varphi_m^k)}\mu_m\nabla\varphi_m,\nabla q_m\right)_m\nonumber\\
&& -\left(\frac{\varpi}{\delta}(\ub_c-\ub_c^k),\mathbf{v}_c\right)_c, \non
 \end{eqnarray}
\begin{equation}
- \left({\rm M}(\vp^k)\nabla \mu,\nabla\phi\right)=\left(\frac{\varphi-\vp^k}{\delta},\phi\right) + (\ub\cdot \nabla \vp,\phi),   \label{wapp2a}
\end{equation}
\begin{equation}
\frac{1}{\epsilon}\left(\varphi^3,\phi\right) +\epsilon(\nabla\varphi,\nabla\phi)=\left(\mu+\frac{1}{\epsilon}\varphi^k,\phi\right). \label{wapp2b}
\end{equation}
As will be shown below, the operator $\mathcal{T}_k : \mathbf{X} \rightarrow \mathbf{Y}$ is continuous and invertible with $\mathcal{T}_k^{-1}(\mathbf{0})=\mathbf{0}$, while the operator $\mathcal{G}_k : \mathbf{X} \rightarrow \mathbf{Y}$ is compact. Thus the abstract equation \eqref{Abs} can be recasted into $$(\mathcal{I}-\mathcal{T}_k^{-1}\mathcal{G}_k)(\mathbf{w})=\mathbf{0}, $$
where $\mathcal{I}: \mathbf{X} \rightarrow \mathbf{X}$ is the identity operator.
  Then the existence of solutions can be shown by Leray-Schauder degree theory.

\begin{remark}
Note that equation \eqref{wapp1} is derived from an addition of a term $\left(\ub_c,\mathbf{v}_c\right)_c$ on both sides of equation \eqref{app1}. This modification is necessary in proving the invertibility of the operator associated with the left-hand side of equation \eqref{wapp1}, especially under the circumstance $|\Gamma_c|=0$ where only the  version \eqref{korn}  of Korn's inequality can be applied.
\end{remark}

We shall divide the proof for the existence of weak solutions to the approximate problem \eqref{app1}--\eqref{app3} into three steps.\medskip

\textbf{Step 1. Invertibility of operators associated with the left-hand sides of the reformulated system \eqref{wapp1}--\eqref{wapp2b}.}

First, we deal with the operator associated with the left-hand side of equation \eqref{wapp1}.
Define the product space
\begin{equation}\label{div-free space}
 \mathbf{V}:=\mathbf{H}_{c,\text{div}}\times X_m.
\end{equation}
Then we  introduce the operator $\mathcal{L}_k: \mathbf{V}\to \mathbf{V}'$ that can be associated with the following bilinear form $a(\cdot, \cdot):\mathbf{V}\times\mathbf{V}\to \mathbb{R}$:
\bea
&&\langle \mathcal{L}_k(\ub_c, P_m), (\mathbf{v}_c, q_m)\rangle_{\mathbf{V}', \mathbf{V}}\non\\
&=& a((\ub_c, P_m), (\mathbf{v}_c, q_m))\non\\
&=&2\left(\nu(\vp_c^k)\mathbb{D}(\ub_c),\mathbb{D}(\mathbf{v}_c)\right)_c+(\ub_c, \mathbf{v}_c)_c+\left(\frac{\Pi}{\nu(\varphi_m^k)}\nabla P_m,\nabla q_m\right)_m
\non\\
&&   +\sum_{i=1}^{d-1}\frac{\alpha_{BJSJ}}{\sqrt{{\rm trace} (\Pi)}} \int_{\Gamma_{cm}}\nu(\vp^k_m)(\ub_c\cdot\btau_i)(\mathbf{v}_c\cdot\btau_i)dS\non\\
&&+\int_{\Gamma_{cm}} P_m (\mathbf{v}_c\cdot \mathbf{n}_{cm}) dS -\int_{\Gamma_{cm}} (\ub_c\cdot \mathbf{n}_{cm})q_m dS,\label{Lk}
\eea
for any $(\ub_c, P_m)$, $(\mathbf{v}_c, q_m) \in \mathbf{V}$.

Recall the following Korn's inequality (cf. e.g., \cite{Hor95}),
\begin{align}
\|\mathbf{v}_c\|_{\mathbf{H}^1(\Omega_c)} \leq C \big(\|\mathbf{v}_c\|_{\mathbf{L}^2(\Omega_c)}+\|\mathbb{D}(\mathbf{v}_c)\|_{\mathbf{L}^2(\Omega_c)}\big), \quad \forall\, \mathbf{v}_c \in \mathbf{H}_{c,\mathrm{div}}, \label{korn}
\end{align}
where the constant $C$ depends only on $\Omega_c$. Moreover, if the boundary $\Gamma_c$ has non-zero measure, the Korn's inequality can be simplified as (cf. e.g., \cite{GiRa1986})
\begin{align}\label{kornsimp}
\|\mathbf{v}_c\|_{\mathbf{H}^1(\Omega_c)} \leq C \|\mathbb{D}(\mathbf{v}_c)\|_{\mathbf{L}^2(\Omega_c)}, \quad \forall\, \mathbf{v}_c \in \mathbf{H}_{c,\mathrm{div}}.
\end{align}
As a consequence, using the assumptions (A1), (A3) and the Poincar\'e inequality, we deduce that the above bilinear form $a(\cdot, \cdot)$ is coercive on $\mathbf{V}$, namely, for any $(\ub_c, P_m) \in \mathbf{V}$,
 \bea
 && a((\ub_c, P_m), (\ub_c, P_m))\non\\
 &=&2\left(\nu(\vp_c^k)\mathbb{D}(\ub_c),\mathbb{D}(\mathbf{u}_c)\right)_c+(\ub_c, \ub_c)_c+\left(\frac{\Pi}{\nu(\varphi_m^k)}\nabla P_m,\nabla P_m\right)_m
\non\\
&& +\sum_{i=1}^{d-1}\frac{\alpha_{BJSJ}}{\sqrt{{\rm trace} (\Pi)}} \int_{\Gamma_{cm}}\nu(\vp^k_m)|\ub_c\cdot\btau_i|^2 dS\non\\
&\geq& C_1\|\ub_c\|_{\mathbf{H}^1(\Omega_c)}^2+C_2\|P_m\|_{H^1(\Omega_m)}^2,\non
 \eea
 for some constants $C_1, C_2$ independent of $\ub_c, P_m$ and $\vp^k$.

Then by the Lax-Milgram lemma, we can easily deduce that
\begin{lemma}\label{LLk}
Assume that the assumptions (A1) and (A3) are satisfied. Then for any given $\vp^k\in H^1(\Omega)$, the operator $\mathcal{L}_k: \mathbf{V}\to \mathbf{V}'$ is invertible and its inverse $\mathcal{L}_k^{-1}: \mathbf{V}' \to \mathbf{V}$ is continuous.
\end{lemma}

Next, we state the invertibility of the operator induced by the left-hand side of equation \eqref{wapp2a}. To this end, we recall the following simple facts in \cite{Ab09}. Define the operator
$\mathrm{div}_N: \mathbf{L}^2(\Omega)\to \dot{H}^{-1}(\Omega)$ by
$$ \langle \mathrm{div}_N \mathbf{v}, \phi\rangle_{\dot{H}^{-1}(\Omega), \dot{H}^1(\Omega)}=-(\mathbf{v}, \nabla \phi), \quad \forall\, \phi \in \dot{H}^1(\Omega).$$
The operator $\mathrm{div}_N$ acts on vector fields, which
do not  necessarily vanish on the boundary, and involves boundary conditions in a weak
sense. Let $\mathrm{M}\in L^\infty(\Omega)$ such that $\mathrm{M}(x)\geq m_0>0$ almost every in $\Omega$.  We then introduce the operator
$\mathrm{div}_N(\mathrm{M}(x) \nabla \cdot): \dot{H}^1(\Omega)\to \dot{H}^{-1}(\Omega)$ defined as
$$ \langle \mathrm{div}_N( \mathrm{M}(x) \nabla \vp), \phi\rangle_{\dot{H}^{-1}(\Omega), \dot{H}^1(\Omega)}=-(\mathrm{M}(x) \nabla \vp, \nabla \phi), \quad \forall\, \phi \in \dot{H}^1(\Omega).$$
Then the operator $\mathrm{div}_N(\mathrm{M}(x) \nabla \cdot)$ is an isomorphism due to an easy application of the Lax-Milgram lemma.

Hence, under the assumption (A2), it is easy to see that
 \begin{lemma}\label{LDk}
 Assume that the function $\mathrm{M}$ satisfies (A2).  For any given $\vp^k\in H^1(\Omega)$, the operator
 \be
 \mathcal{D}_k:=\mathrm{div}_N(\mathrm{M}(\vp^k) \nabla \cdot): \dot{H}^1(\Omega)\to \dot{H}^{-1}(\Omega)\label{Dk}
 \ee
  is invertible and its inverse $\mathcal{D}_k^{-1}: \dot{H}^{-1}(\Omega) \to \dot{H}^1(\Omega)$ is continuous.
\end{lemma}

We now proceed to the solvability of equation \eqref{wapp2b}.
 For any given function $\varphi^k\in H^1(\Omega)$, we define the nonlinear operator $\mathcal{S}_k: \dot{H}^1(\Omega) \rightarrow \dot{H}^{-1}(\Omega)$ as follows
\begin{align}\label{2ndOp}
\langle \mathcal{S}_k(\psi), \phi \rangle_{\dot{H}^{-1}(\Omega), \dot{H}^1(\Omega)}
=\epsilon(\nabla \psi, \nabla \phi)+\frac{1}{\epsilon}\big((\psi+\overline{\vp^k})^3, \phi\big), \quad \forall \phi \in \dot{H}^1(\Omega),
\end{align}
where $\overline{\vp^k}=|\Omega|^{-1}\int_\Omega \vp^k dx$.

Then we have

\begin{lemma}\label{LSk}
Let $\varphi^k\in H^1(\Omega)$ be fixed. For any given function $\mu_0 \in \dot{H}^{-1}(\Omega)$, there exists a unique solution $\psi \in \dot{H}^1(\Omega)$ to the equation $\mathcal{S}_k(\psi)=\mu_0$. The solution operator $\mathcal{S}_k^{-1}: \dot{H}^{-1}(\Omega) \rightarrow  \dot{H}^1(\Omega)$ is continuous. Moreover, if $\mu_0 \in \dot{H}^1(\Omega)$, then the solution satisfies $\psi \in \dot{H}^3(\Omega)$ and $\mathcal{S}_k^{-1}: \dot{H}^{1}(\Omega) \rightarrow  \dot{H}^3(\Omega)$ is bounded and continuous.
\end{lemma}
\begin{proof}
The unique solvability of equation $\mathcal{S}_k(\psi)=\mu_0$ for given source function $\mu_0$ can be obtained by the theory of monotone operators.

We note that $\mathcal{S}_k$ is well defined for any given function $\varphi^k \in H^1(\Omega)$. Indeed, using the Sobolev embedding $H^1(\Omega)\hookrightarrow L^6(\Omega)$ for $d=2, 3$, we can see that for any $\psi\in \dot{H}^1(\Omega)$,
\be \Big|\langle \mathcal{S}_k(\psi), \phi \rangle_{\dot{H}^{-1}(\Omega), \dot{H}^1(\Omega)}\Big|\leq C(\epsilon)\left(\|\psi\|_{H^1(\Omega)}^3+|\overline{\vp^k}|^3+\|\psi\|_{H^1(\Omega)}\right)\|\phi\|_{H^1(\Omega)},\non
\ee
 which implies the boundedness of $\mathcal{S}_k$ in $H^1(\Omega)$.
  Moreover, if a sequence $\psi_n \rightarrow \psi $ in $\dot{H}^1(\Omega)$ as $n\rightarrow \infty$, by H\"{o}lder's inequality and the  Sobolev embedding, we deduce that for any  $ \phi \in \dot{H}^1(\Omega)$,
\bea
&& \Big|\langle \mathcal{S}_k(\psi_n)-\mathcal{S}_k(\psi), \phi \rangle_{\dot{H}^{-1}(\Omega), \dot{H}^1(\Omega)}\Big|\non\\
 &\leq&
 C(\epsilon) \left(\|(\psi_n+\overline{\vp^k})^3-(\psi+\overline{\vp^k})^3\|_{L^\frac{6}{5}(\Omega)}
 +\|\nabla(\psi_n-\psi)\|_{\mathbf{L}^2(\Omega)}\right)\|\phi\|_{H^1(\Omega)}\non\\
 &\leq&  C(\epsilon)\big(\|\psi_n^2+\psi^2+(\overline{\vp^k})^2\|_{L^2(\Omega)}
 \|\psi_n-\psi\|_{L^3(\Omega)}+\|\nabla(\psi_n-\psi)\|_{\mathbf{L}^2(\Omega)}\big)
 \|\phi\|_{H^1(\Omega)}\non\\
 & \rightarrow& 0.\non
\eea
Hence, the nonlinear operator $\mathcal{S}_k: \dot{H}^1(\Omega)\rightarrow \dot{H}^{-1}(\Omega)$ is continuous. Concerning the coercivity of $\mathcal{S}_k$, using the Young inequality,  we have for any $\psi\in \dot{H}^1(\Omega)$,
\begin{eqnarray}
&& \langle \mathcal{S}_k(\psi), \psi \rangle_{\dot{H}^{-1}(\Omega), \dot{H}^1(\Omega)}\non\\
& =& \frac{1}{\epsilon}\int_\Omega (\psi+\overline{\vp^k})^3\psi\, dx 
+\epsilon \int_\Omega |\nabla \psi|^2\,dx \nonumber \\
&\geq& \frac{1}{\epsilon}\int_\Omega |\psi|^4\, dx-\frac{3|\overline{\vp^k}|}{\epsilon}\int_\Omega |\psi|^3 dx  -\frac{3|\overline{\vp^k}|^2}{\epsilon}\int_\Omega |\psi|^2 dx -\frac{|\overline{\vp^k}|^3}{\epsilon}\int_\Omega |\psi|dx\non\\
&& 
+\epsilon \int_\Omega |\nabla \psi|^2\,dx\nonumber \\
&\geq & C(\epsilon)\|\psi\|_{H^1(\Omega)}^2-C(\epsilon,|\Omega|,|\overline{\vp^k}|), \label{2ndCoer}
\end{eqnarray}
which yields that
\begin{align*}
\frac{\langle \mathcal{S}_k(\psi), \psi \rangle_{\dot{H}^{-1}(\Omega), \dot{H}^1(\Omega)} }{\|\psi\|_{H^1(\Omega)}} \rightarrow +\infty, \quad \text{ as } \|\psi\|_{H^1(\Omega)} \rightarrow \infty.
\end{align*}
Finally, the strict monotonicity of $\mathcal{S}_k$ follows from the following identity
\begin{eqnarray}\label{2ndMono}
&& \langle \mathcal{S}_k(\psi_1)-\mathcal{S}_k(\psi_2), \psi_1-\psi_2 \rangle_{\dot{H}^{-1}(\Omega), \dot{H}^1(\Omega)} \non\\
&=&\frac{1}{\epsilon}\int_\Omega (\psi_1-\psi_2)^2\left[(\psi_1+\overline{\vp^k})^2+(\psi_2+\overline{\vp^k})^2+(\psi_1+\overline{\vp^k})(\psi_2+\overline{\vp^k})\right]\,dx \non\\
&& + \epsilon\int_\Omega |\nabla(\psi_1-\psi_2)|^2\,dx\non\\
&\geq& 0,\quad \forall\, \psi_1, \psi_2 \in \dot{H}^1(\Omega)
\end{eqnarray}
and the equal sign holds if and only if $\psi_1=\psi_2$.

Based on the above observations, we can apply the Browder-Minty theorem (cf. e.g., \cite[pp. 39, Theorem 2.2]{Showalter1997}) to conclude the existence of a unique solution $\psi \in \dot{H}^1(\Omega)$ to the nonlinear equation  $\mathcal{S}_k(\psi)=\mu_0$ for a given source function $\mu_0 \in \dot H^{-1}(\Omega)$.
The coercive estimate \eqref{2ndCoer} also implies that
\begin{eqnarray}\label{BoundPhiH1}
\|\psi\|_{H^1(\Omega)}^2 \leq  C(\epsilon)\left(\|\mu_0\|_{\dot H^{-1}(\Omega)}^2+|\overline{\vp^k}|^4+1\right).
\end{eqnarray}
For the continuous dependence of the solution $\psi$ on $\mu_0$, i.e., if a sequence $\mu_{0n} \rightarrow \mu_0$ strongly in $\dot{H}^{-1}(\Omega)$ and $\mathcal{S}_k(\psi_n)=\mu_{0n}$, $\mathcal{S}_k(\psi)=\mu_0$, then $\psi_n, \psi\in \dot{H}^1(\Omega)$ and as $n\to +\infty$, it holds
\be
\langle \mathcal{S}_k(\psi_n)-\mathcal{S}_k(\psi), \psi_n-\psi \rangle_{\dot{H}^{-1}(\Omega), \dot{H}^1(\Omega)}=
\langle \mu_{0n}-\mu_0, \psi_n-\psi \rangle_{\dot{H}^{-1}(\Omega), \dot{H}^1(\Omega)}\to 0.
\ee
 Then a similar estimate like \eqref{2ndMono} yields that $\psi_n\to \psi$ strongly in $\dot{H}^1(\Omega)$. As a consequence, the solution operator $\mathcal{S}_k^{-1}: \dot{H}^{-1}(\Omega) \rightarrow  \dot{H}^1(\Omega)$ is continuous.

If  we further assume that $\mu_0\in \dot{H}^1(\Omega)$, the weak solution $\psi$ indeed  has higher regularity. To this end, we rewrite the weak form of the equation $\mathcal{S}_k(\psi)=\mu_0$ as
\begin{align}
\epsilon \big(\nabla \psi, \nabla \phi\big)=\Big(\mu_0-G(\psi,\vp^k), \phi\Big), \quad \forall \phi \in \dot{H}^1(\Omega),\non
\end{align}
where $G(\psi, \vp^k)=\epsilon^{-1}(\psi+\overline{\vp^k})^3\in L^2(\Omega)$.
Then $\psi$ is a weak solution to the following elliptic equation with homogeneous Neumann boundary condition:
\begin{equation}\label{p2ndStr}
\begin{cases}
&-\epsilon \Delta \psi=\mu_0-G_0, \quad \mbox{in}\ \Omega, \\
& \quad \displaystyle{\frac{\partial \psi}{\partial \mathbf{n}}}=0,\ \ \mbox{on}\ \partial \Omega,\\
& \int_\Omega \psi dx=0,
\end{cases}
\end{equation}
with $G_0=G(\psi, \vp^k)-\overline{G}(\psi,\vp^k)$.
Since the source function $\mu_0-G_0\in \dot{L}^2(\Omega)$, one deduces from the classical elliptic regularity theory (cf. \cite{Grisvard1985}) that $\psi \in H^2(\Omega)$ if $\Omega$ is $C^{1,1}$ or a convex bounded domain. In particular, one can derive from \eqref{p2ndStr} that
\begin{align}\label{p2ndLapH2}
\|\psi\|_{H^2(\Omega)} \leq C(\epsilon)\left(\|\mu_0\|_{L^2(\Omega)}+\|\psi\|_{H^1(\Omega)}^3
+|\overline{\vp^k}|^3+\|\psi\|_{L^2(\Omega)}\right),
\end{align}
Since $H^2(\Omega)$ is an algebra with respect to point-wise multiplication in $\mathbb{R}^d$ $(d\leq 3)$, one has $\mu_0-G_0 \in \dot H^1(\Omega)$. Then it follows from \eqref{p2ndStr}, \eqref{p2ndLapH2}  that
\begin{eqnarray}\label{p2ndLapH3}
\|\psi\|_{H^3}&\leq& C(\epsilon)\left(\|\mu_0\|_{H^1(\Omega)}+|\overline{\vp^k}|^3+\|\psi^3\|_{H^1(\Omega)}+\|\psi\|_{L^2(\Omega)}\right) \nonumber \\
&\leq & C(\epsilon)\left(\|\mu_0\|_{H^1(\Omega)}+|\overline{\vp^k}|^3+\|\psi\|_{L^2(\Omega)}\right)\nonumber\\
&& +C(\epsilon)\left(\|\psi\|_{L^\infty(\Omega)}^2 \|\nabla \psi\|_{\mathbf{L}^2(\Omega)}+\|\psi\|_{L^6(\Omega)}^3\right) \nonumber \\
&\leq & C(\epsilon, \Omega, \|\mu_0\|_{H^1(\Omega)}, |\overline{\vp^k}|),
\end{eqnarray}
which yields that the solution operator $\psi=\mathcal{S}_k^{-1}(\mu_0)$ is bounded from $\dot{H}^1(\Omega)$ to $\dot{H}^3(\Omega)$. Consider the difference problem
\begin{equation}\label{p2ndStrd}
\begin{cases}
&-\epsilon \Delta (\psi_n-\psi)=(\mu_{0n}-\mu_0)-(G_{0n}-G_0), \\
& \quad \displaystyle{\frac{\partial (\psi_n-\psi)}{\partial \mathbf{n}}}=0,\ \ \mbox{on}\ \partial \Omega,
\end{cases}
\end{equation}
with $G_{0n}=G(\psi_n, \vp^k)-\overline{G}(\psi_n,\vp^k)$ and $G_0=G(\psi, \vp^k)-\overline{G}(\psi,\vp^k)$.
Assuming that $\mu_{0n} \rightarrow \mu_0$ strongly in $\dot{H}^1(\Omega)$, similar to \eqref{p2ndLapH2}, we can first derive the $H^2$ estimates for $\psi_n$, $\psi$, and then use the elliptic estimates as in \eqref{p2ndLapH3} to get
\bea
\|\psi_n-\psi\|_{H^3(\Omega)}
&\leq& C(\|\mu_{0n}-\mu_0\|_{H^1(\Omega)}+\|G_{0n}-G_0\|_{H^1(\Omega)}+\|\psi_n-\psi\|_{L^2(\Omega)})\non\\
&\leq&  C(\|\psi_n\|_{L^\infty(\Omega)}, \|\psi\|_{L^\infty(\Omega)}, \|\nabla \psi_n\|_{\mathbf{L}^3(\Omega)}, \|\nabla \psi\|_{\mathbf{L}^3(\Omega)})\|\psi_n-\psi\|_{H^1(\Omega)}\non\\
&& +C\|\mu_{0n}-\mu_0\|_{H^1(\Omega)}.\non
\eea
We have already shown that $\mathcal{S}_k^{-1}: \dot{H}^{-1}(\Omega) \rightarrow  \dot{H}^1(\Omega)$ is continuous, which  combining the above estimate further yields  that $\mathcal{S}_k^{-1}: \dot{H}^{1}(\Omega) \rightarrow  \dot{H}^3(\Omega)$ is also (strongly) continuous. The proof is complete.
\end{proof}

\textbf{Step 2. Definition of operators $\mathcal{T}_k$, $\mathcal{G}_k$ and their properties.}

We introduce the following product spaces
\be
\begin{cases}
 \mathbf{X}=\mathbf{V}\times \dot{H}^1(\Omega)\times \dot{H}^{3-\sigma}(\Omega)\times \mathbb{R},\\
 \mathbf{Y}=\mathbf{V}'\times \dot{H}^{-1}(\Omega)\times \dot{L}^{2}(\Omega)\times \mathbb{R},
 \end{cases}
 \label{XY}
\ee
where $\sigma\in (0, \frac12)$ is a constant.

 Owing to the mass-conservation property \eqref{MC} of the approximate scheme and for the convenience of the norm of $\dot{H}^1(\Omega)$, we will project the unknowns $\varphi$ and $\mu$ into $\dot{L}^2(\Omega)$ such that $$ \varphi=\psi+\overline{\varphi^k}, \quad \mu=\mu_0+ \overline{S}_k, $$
where $\overline{\varphi^k}$ and $\overline{S}_k$ are the average of $\varphi_k$ and $\widetilde{f}(\varphi,\varphi^k)$ on $\Omega$, respectively.

According to the formulation of the system \eqref{wapp1}--\eqref{wapp2b}, we now introduce the nonlinear operators $\mathcal{T}_k$, $\mathcal{G}_k: \mathbf{X}\to \mathbf{Y}$.
 For any given functions $\varphi^k\in H^1(\Omega)$, $\ub_c^k\in \mathbf{L}^2(\Omega_c)$ and for $\mathbf{w}=(\ub_c, P_m, \mu_0, \psi, \overline{S}_{k})\in \mathbf{X}$,  we define
\begin{equation}
\mathcal{T}_k (\mathbf{w}) =
\left(
  \begin{array}{c}
    \mathcal{L}_k(\ub_c,P_m)\\
    \mathcal{D}_k (\mu_0)\\
    \mathcal{S}_k (\psi)\\
    \overline{S}_{k}
  \end{array}
\right),   \label{Tk}
\end{equation}
and
\begin{equation}
\mathcal{G}_k (\mathbf{w}) =
\left(
  \begin{array}{c}
   \mathcal{J}_k(\mathbf{w}) \\
    \mathrm{P}_0\Big({\delta}^{-1}(\psi+\overline{\vp^k}-\vp^k)+\ub\cdot \nabla \psi\Big)\\
    \mu_0+\epsilon^{-1}(\vp^k-\overline{\vp^k})\\
    |\Omega|^{-1}\epsilon^{-1}\int_{\Omega} \Big((\psi+\overline{\vp^k})^3-\vp^k\Big)dx
  \end{array}
\right),  \label{Gk}               %
\end{equation}
  The operators $\mathcal{L}_k$, $\mathcal{D}_k$, $\mathcal{S}_k$ in \eqref{Tk} are defined in \eqref{Lk}, \eqref{Dk} and \eqref{2ndOp} (associated with the given function $\varphi^k$), respectively. In \eqref{Gk}, the operator $\mathcal{J}_k: \mathbf{X}\to \mathbf{V}'$ is given by
 \bea
 && \langle \mathcal{J}_k(\mathbf{w}), (\mathbf{v}_c, q_m) \rangle_{\mathbf{V}', \mathbf{V}}\non\\
 &=& \left(-\displaystyle{\frac{\varpi}{\delta}}(\ub_c-\ub_c^k) +(\mu_{0c}+ \overline{S}_{k}) \nabla\psi_c, \ \mathbf{v}_c\right)_c+(\ub_c, \mathbf{v}_c)_c\non\\
 && + \left( \displaystyle{\frac{\Pi}{\nu(\vp_m^k)}}(\mu_{0m}+ \overline{S}_{k})\nabla\psi_m, \nabla q_m\right)_m, \quad \forall \, (\mathbf{v}_c, q_m) \in \mathbf{V}. \label{Jk}
 \eea
Here, one recalls that $P_0$ is the projection operator from $L^2(\Omega)$ into $\dot{L}^2(\Omega)$ and the facts $\mu_{0c}=\mu_0|_{\Omega_c}$, $\mu_{0m}=\mu_0|_{\Omega_m}$. The velocity $\ub$ in \eqref{Gk} fulfills $\ub|_{\Omega_c}=\ub_c$, $\ub|_{\Omega_m}=\ub_m$ and $\ub_m$ is given by \eqref{app3}.

From the definition of $\mathcal{T}_k$ and Lemmas \ref{LLk}--\ref{LSk} obtained in the previous step, one can conclude that
\begin{lemma}\label{LTG}
    $\mathcal{T}_k: \mathbf{X}\to\mathbf{Y}$ is an invertible mapping and its inverse $\mathcal{T}^{-1}_k: \mathbf{Y}\to \mathbf{X}$ is continuous. In particular, $\mathcal{T}_k^{-1}(\mathbf{0})=\mathbf{0}$.
\end{lemma}

Then concerning the operator $\mathcal{G}_k$, one has
\begin{lemma}\label{LGk}
    $\mathcal{G}_k: \mathbf{X}\to\mathbf{Y}$ is a continuous and bounded mapping. Moreover, it is compact.
\end{lemma}
\begin{proof}
For all $\mathbf{w}=(\ub_c, P_m,  \mu_0, \psi, \overline{S}_{k})\in \mathbf{X}$, using the Sobolev embedding theorems  $(d\leq 3)$ such that $H^1\hookrightarrow L^6$ and $H^{1-\sigma}\hookrightarrow L^3$, $H^{2-\sigma}\hookrightarrow L^\infty$ for $\sigma\in (0, \frac12)$, it is straightforward to show that
\be
\mathcal{G}_k (\mathbf{w})\in \left(\mathbf{L}^2(\Omega_c)
\times (H^{1-\sigma}(\Omega_m))'\right)\times \dot{L}^2 (\Omega)\times \dot{H}^1(\Omega)\times K \hookrightarrow\hookrightarrow \mathbf{Y},\non
\ee
where $K$ is a bounded set in $\mathbb{R}$. Our conclusion easily follows.
\end{proof}

We now interpret the relation between the abstract equation $\mathcal{T}_k(\mathbf{w})=\mathcal{G}_k(\mathbf{w})$ for $\mathbf{w} \in \mathbf{X}$ and the elliptic system \eqref{app1}--\eqref{app2b}. The following equivalence result can be easily seen from the definitions \eqref{Lk}--\eqref{2ndOp} and \eqref{Tk}--\eqref{Jk}:
\begin{proposition}\label{Abequi}
 $\{\ub_c, P_m, \varphi, \mu\}\in  \mathbf{H}_{c,{\rm div}} \times X_m\times  H^3(\Omega)\times  H^1(\Omega)$ is a solution of the system \eqref{app1}--\eqref{app2b} if and only if
$\mathbf{w}=(\ub_c, P_m, \mu_0, \psi, \overline{S}_k) \in \mathbf{X}$ satisfies
$\mathcal{T}_k(\mathbf{w})=\mathcal{G}_k(\mathbf{w})$ with $\vp=\psi+ \overline{\vp^k}$, $\mu=\mu_{0}+\overline{S}_{k}$.
\end{proposition}

\textbf{Step 3. Solvability of the nonlinear system \eqref{app1}--\eqref{app3}}

We proceed to show that there exists a $\mathbf{w} \in \mathbf{X}$ such that $\mathcal{T}_k(\mathbf{w})=\mathcal{G}_k(\mathbf{w})$. Since $\mathcal{T}_k$ is invertible, this abstract equation can be rewritten equivalently as $\mathbf{w}=\mathcal{T}_k^{-1}(\mathcal{G}_k(\mathbf{w}))$, namely,
\be
(\mathcal{I}-\mathcal{N}_k)(\mathbf{w})=\mathbf{0}.\label{IaNK}
\ee
  where $\mathcal{I}$ is the identity operator on $\mathbf{X}$ and the nonlinear operator $\mathcal{N}_k$ is defined by
\be
\mathcal{N}_k(\mathbf{w}):=\mathcal{T}_k^{-1}(\mathcal{G}_k(\mathbf{w})): \mathbf{X}\to \mathbf{X}, \quad \forall\, \mathbf{w}\in \mathbf{X}\label{NNN}
 \ee
 and it is a compact operator on $\mathbf{X}$ due to Lemmas \ref{LTG} and \ref{LGk}. Thus we only have to prove that there exists a vector $\mathbf{w}=(\ub_c, P_m, \mu_0, \psi, \overline{S}_k)\in \mathbf{X}$ that satisfies equation \eqref{IaNK}. This can be done by a homotopy argument based on the Leray-Schauder degree (cf. \cite{De, Ab09}).
\begin{lemma}\label{degree}
Assume that assumptions (A1)--(A3) are satisfied. For any $\ub_c^k\in \mathbf{L}^2(\Omega_c)$ and $\vp^k\in H^1(\Omega)$, the equation $\mathcal{T}_k(\mathbf{w})=\mathcal{G}_k(\mathbf{w})$ admits a solution  $\mathbf{w}=(\ub_c, P_m, \mu_0, \psi, \overline{S}_k)\in \mathbf{X}$.
\end{lemma}
\begin{proof}
For $s\in [0,1]$, we define
$$  \ub_c^k(s)=(1-s)\ub_c^k,\qquad  \varphi^k(s)=(1-s)\varphi^k.$$
Replace $\ub_c^k$, $\varphi^k$ in the system \eqref{wapp1}--\eqref{wapp2b} by $\ub_c^k(s)$, $\varphi^k(s)$, respectively. Then we denote by $\mathcal{T}_k^{(s)}$, $\mathcal{G}_k^{(s)}$ the corresponding operators under the above transformation. In particular, $\mathcal{T}_k^{(0)}=\mathcal{T}_k$, $\mathcal{G}_k^{(0)}=\mathcal{G}_k$.
It is easy to see that $\mathcal{T}_k^{(s)}$, $\mathcal{G}_k^{(s)}$ have the same properties as in Lemmas \ref{LTG}--\ref{LGk}. Then we denote by $\mathcal{N}_k^{(s)}=(\mathcal{T}_k^{(s)})^{-1}\mathcal{G}_k^{(s)}$, which is a compact operator. Moreover, $\mathcal{N}_k^{(0)}=\mathcal{N}_k$.

In analogy  to \eqref{DisEnLaw}, we can derive the following discrete energy law with respect to the parameter $s$:
\begin{eqnarray}\label{DisEnLawa}
&&
  \mathcal{E}(\ub_c,\varphi) +\delta \left( \nu(\varphi_m^k(s))\Pi^{-1}\ub_m, \ub_m\right)_m
  +2\delta \left(\nu(\vp_c^k(s))\mathbb{D}(\ub_c), \mathbb{D}(\ub_c)\right)_c
\non\\
&& +\delta\int_\Omega \mathrm{M}(\vp^k(s))|\nabla\mu|^2 dx+\frac{\delta \alpha_{BJSJ}}{\sqrt{{\rm trace}(\Pi)}}\sum_{i=1}^{d-1}\int_{\Gamma_{cm}} \nu(\vp^k_m(s))|\ub_c\cdot\btau_i|^2 dS
  \non\\
  &&+\frac{\varpi}{2}\left(\ub_c-\ub_c^k(s) ,\ub_c-\ub_c^k(s)\right)_c +\frac{\epsilon}{2}\|\nabla (\varphi-\varphi^k(s))\|_{\mathbf{L}^2(\Omega)}^2\non\\
  &&+ \frac{1}{2\epsilon}\|\vp-\vp^k(s)\|_{L^2(\Omega)}^2\non\\
&\le& \mathcal{E}(\ub_c^k(s),\varphi^k(s)).
\end{eqnarray}

For any given $\ub_c^k\in \mathbf{L}^2(\Omega_c)$ and $\vp^k\in H^1(\Omega)$, there exists a constant $R>0$ depending only on $\|\ub_c^k\|_{\mathbf{L}^2(\Omega_c)}$, $\|\vp^k\|_{H^1(\Omega)}$, $\varpi$, $\epsilon$  and $\Omega$ such that $\mathcal{E}(\ub_c^k(s),\varphi^k(s))\leq R$ for all $s\in [0,1]$. By the energy estimate \eqref{DisEnLawa}, there exists $C_0>0$ depending on $R$ and coefficients of the system but independent of $s$ such that
the solution $\mathbf{w}=\mathbf{w}^{(s)}$ to the equation $\mathcal{T}_k^{(s)}(\mathbf{w})=\mathcal{G}_k^{(s)}(\mathbf{w})$, if it exists, will satisfy $$\|\mathbf{w}^{(s)}\|_{\mathbf{X}}\leq C_0,\quad \forall\, s\in [0,1].$$
 Taking the ball in $\mathbf{X}$ centered at $\mathbf{0}$ with radius $2C_0$:
$$\mathbf{B}=\{ \mathbf{w}\in \mathbf{X}:\ \|\mathbf{w}\|_{\mathbf{X}}\leq 2C_0\},$$
 we infer from the above \textit{a priori} estimate that for all $s\in [0,1]$, $(\mathcal{I}-\mathcal{N}_k^{(s)})(\mathbf{w})\neq \mathbf{0}$ for any $\mathbf{w}\in \partial \mathbf{B}$. Therefore,  the Leray-Schauder degree of the operator $\mathcal{I}-\mathcal{N}_k^{(s)}$ at $\mathbf{0}$ in the ball $\mathbf{B}$, denoted by $\mathrm{deg}(\mathcal{I}-\mathcal{N}_k^{(s)}, \mathbf{B}, \mathbf{0})$, is well-defined for $s\in [0,1]$.

On the other hand, since $\mathcal{N}_k^{(0)}=\mathcal{N}_k$, then by the homotopy invariance of the Leray-Schauder degree, we have
\be
\mathrm{deg}(\mathcal{I}-\mathcal{N}_k, \mathbf{B}, \mathbf{0})=\mathrm{deg}(\mathcal{I}-\mathcal{N}_k^{(0)}, \mathbf{B}, \mathbf{0})=\mathrm{deg}(\mathcal{I}-\mathcal{N}_k^{(1)}, \mathbf{B}, \mathbf{0}).\label{degg1}
\ee

Next, we shall prove that $\mathrm{deg}(\mathcal{I}-\mathcal{N}_k^{(1)}, \mathbf{B}, \mathbf{0})=1$. For this purpose, we define a further homotopy for $s\in [1,2]$ such that
\be
\mathcal{N}_k^{(s)}(\mathbf{w})=\left(\mathcal{T}_k^{(1)}\right)^{-1}\left[(2-s)\mathcal{G}_k^{(1)}(\mathbf{w})\right], \quad \forall\, \mathbf{w}\in \mathbf{X}.
\ee
For $s\in [1, 2)$, $(\mathcal{I}-\mathcal{N}_k^{(s)})(\mathbf{w})=\mathbf{0}$ if and only if for  $\mathbf{w}=(\ub_c, P_m, \mu_0, \psi, \overline{S}_k)\in \mathbf{X}$, the vector $(\ub_c, P_m, \vp, \mu)$  with $\vp=\psi$, $\mu=\mu_0+\overline{S}_k (2-s)^{-2}$  is a solution of the following system
\begin{eqnarray}
&&\frac{\varpi (2-s)}{\delta}\left(\ub_c ,\mathbf{v}_c\right)_c+2\left(\nu(0)\mathbb{D}(\ub_c),\mathbb{D}(\mathbf{v}_c)\right)_c
\non\\
&& +(s-1)(\ub_c, \mathbf{v}_c)_c+\left(\frac{\Pi}{\nu(0)}\nabla P_m,\nabla q_m\right)_m\non\\
 && +\sum_{i=1}^{d-1}\frac{\alpha_{BJSJ}}{\sqrt{{\rm trace} (\Pi)}} \int_{\Gamma_{cm}}\nu(0)(\ub_c\cdot\btau_i)(\mathbf{v}_c\cdot\btau_i)dS\non\\
&& +\int_{\Gamma_{cm}} P_m (\mathbf{v}_c\cdot \mathbf{n}_{cm}) dS-\int_{\Gamma_{cm}} (\ub_c\cdot \mathbf{n}_{cm})q_m dS \non\\
&=&(2-s)(\mu_c\nabla\varphi_c,\mathbf{v}_c)_c+(2-s)\left(\frac{\Pi}{\nu(0)}\mu_m\nabla\varphi_m,\nabla q_m\right)_m,\label{app1q}
 \end{eqnarray}
\be
\frac{2-s}{\delta}\left(\varphi,\phi\right) + (2-s)(\ub\cdot \nabla \vp ,\phi)=- ({\rm M}(0)\nabla \mu,\nabla\phi),\label{app2aq}
\ee
\be
(2-s)(\mu,\phi)=\frac{1}{\epsilon}\left(\varphi^3,\phi\right) +\epsilon(\nabla\varphi,\nabla\phi), \label{app2bq}
\ee
for any $q_m \in X_m$,  $\mathbf{v}_c\in \mathbf{H}_{c,{\rm div}}$, $\phi\in H^1(\Omega)$, and $\ub_m$ is given by
\be
\ub_m=-\frac{\Pi}{\nu(0)}\left[\nabla P_m-\mu(\varphi_m)\nabla\varphi_m\right].\label{app3q}
 \ee
%
Taking the testing functions $\mathbf{v}_c=\ub_c$, $q_m =P_m$ in \eqref{app1q}, $\phi=\mu$ in \eqref{app2aq} and $\phi=\vp$ in \eqref{app2bq}, summing up, we obtain that
\begin{eqnarray}
&&\frac{\varpi (2-s)}{\delta}\left(\ub_c ,\mathbf{u}_c\right)_c+\frac{\epsilon}{\delta}(\nabla\varphi,\nabla\varphi)+\frac{1}{\delta\epsilon}\int_\Omega  \varphi^4dx \non\\
&& +2\left(\nu(0)\mathbb{D}(\ub_c),\mathbb{D}(\mathbf{u}_c)\right)_c+(s-1)(\ub_c, \mathbf{u}_c)_c\non\\
&& +\left(\frac{\Pi}{\nu(0)}\nabla P_m,\nabla P_m\right)_m+\sum_{i=1}^{d-1}\frac{\alpha_{BJSJ}}{\sqrt{{\rm trace} (\Pi)}} \int_{\Gamma_{cm}}\nu(0)|\ub_c\cdot\btau_i|^2 dS\non\\
&&+ ({\rm M}(0)\nabla \mu,\nabla\mu)
\non\\
&=&0.
 \end{eqnarray}
 The above estimate implies that for $s\in (1,2)$, $(\mathcal{I}-\mathcal{N}_k^{(s)})(\mathbf{w})=\mathbf{0}$  if and only if $\mathbf{w}=\mathbf{0}$. Moreover, it is straightforward to check that $\mathcal{I}-\mathcal{N}_k^{(2)}=\mathcal{I}$ (cf. Lemmas \ref{LTG}, \ref{LGk}, in particular, $\big(\mathcal{T}_k^{(1)}\big)^{-1}(\mathbf{0})=\mathbf{0}$) and thus $(\mathcal{I}-\mathcal{N}_k^{(2)})(\mathbf{w})=\mathbf{0}$ if and only if $\mathbf{w}=\mathbf{0}$. Thus, for $s\in [1,2]$, we have
  $(\mathcal{I}-\mathcal{N}_k^{(s)})(\mathbf{w})\neq \mathbf{0}$ for any $\mathbf{w}\in \partial \mathbf{B}$. As a consequence, the homotopy invariance of the Leray-Schauder degree yields that
   \be
   \mathrm{deg}(\mathcal{I}-\mathcal{N}_k^{(1)}, \mathbf{B}, \mathbf{0})=\mathrm{deg}(\mathcal{I}, \mathbf{B}, \mathbf{0})=1.\label{degg2}
   \ee

 In summary, we can conclude from \eqref{degg1} and \eqref{degg2} that $\mathrm{deg}(\mathcal{I}-\mathcal{N}_k, \mathbf{B}, \mathbf{0})=1$, which implies that the abstract equation \eqref{IaNK} admits a solution $\mathbf{w}=(\ub_c, P_m,  \mu_0, \psi, \overline{S}_k)\in \mathbf{X}$ that solves $\mathcal{T}_k(\mathbf{w})=\mathcal{G}_k(\mathbf{w})$.

 The proof of Lemma \ref{degree} is complete.
\end{proof}

Finally, we can conclude the existence of weak solutions to the system \eqref{app1}--\eqref{app2b} from Lemmas \ref{LMC}, \ref{DEE}, \ref{LSk}, \ref{degree} and Proposition \ref{Abequi},
\begin{theorem}[Existence of solutions to the discrete problem]\label{DEEl}
For every $\ub_c^k\in \mathbf{L}^2(\Omega_c)$ 
and $\vp^k\in H^1(\Omega)$,  there exists a weak solution $\{\ub_c, \ub_m, P_m, \varphi, \mu\}$ to the nonlinear discrete problem \eqref{app1}--\eqref{app3} such that
$$
 \ub_c \in \mathbf{H}_{c,{\rm div}}, \quad \mathbf{u}_m \in \mathbf{H}_{m, \mathrm{div}}, \quad P_m\in X_m, \quad \varphi \in H^3(\Omega), \quad \mu \in H^1(\Omega).
$$
Moreover, the solution satisfies the mass-conservation property \eqref{MC} and the energy-dissipation inequality \eqref{DisEnLaw}.
\end{theorem}



\subsection{Construction of approximate solution and passage to limit}
The existence of weak solutions to the time-discrete system \eqref{app1}--\eqref{app3}  enables us to construct approximate solutions to the time-continuous system \eqref{weak1}--\eqref{weak4}. Recall that $\delta= \frac{T}{N}$, where $T>0$ and $N$ is an positive integer. We set $$t_k=k \delta, \quad k=0, 1,\cdots, N.$$
Let $\{\ub_c^{k+1}, P_m^{k+1}, \varphi^{k+1}, \mu^{k+1}\}$ ($k=0, 1,\cdots, N-1$) be chosen successively as a solution of the discretized problem \eqref{app1}--\eqref{app3} with $(\ub_c^{k}, \vp^{k})$ being the ``initial value" (cf. Theorem \ref{DEEl}). In particular, $(\ub_c^{0}, \vp^{0})=(\ub_0, \vp_0)$.
Then for $k=0,1, \cdots, N-1$, we define the approximate solutions as follows
\begin{align*}
&\varphi^\delta :=\frac{t_{k+1}-t}{\delta}\varphi^{k}+\frac{t-t_k}{\delta}\varphi^{k+1}, \quad \text{for}\ t \in [t_k, t_{k+1}],  \\
&\mathbf{u}^\delta_c:= \frac{t_{k+1}-t}{\delta}\ub_c^{k}+\frac{t-t_k}{\delta}\ub_c^{k+1}, \quad \text{for}\ t \in [t_k, t_{k+1}],\\
&\widehat{P}_m^\delta:=  P_m^{k+1},\quad \text{for}\ t \in (t_k, t_{k+1}],\\
& \widehat{\ub}_m^\delta := -\frac{\Pi}{\nu(\varphi_m^k)}\left(\nabla P_m^{k+1}-\mu^{k+1}\nabla\varphi_m^{k+1}\right),\quad \text{for}\  t \in (t_k, t_{k+1}],\\
& \widehat{\ub}_c^\delta: =\ub_c^{k+1},\quad \text{for}\  t \in (t_k, t_{k+1}],\\
& \widehat{\mathbf{u}}^\delta|_{\Omega_c}= \widehat{\ub}_c^\delta, \quad \widehat{\mathbf{u}}^\delta|_{\Omega_m}=\widehat{\mathbf{u}}^\delta_m,\quad \text{for}\  t \in (t_k, t_{k+1}],\\
&\widehat{\varphi}^\delta := \varphi^{k+1},  \quad \text{for}\  t \in (t_k, t_{k+1}],  \\
&\widetilde{\varphi}^\delta := \varphi^{k}, \qquad \text{for}\  t \in [t_k, t_{k+1}), \\
&\widehat{\mu}^\delta: = \mu^{k+1},  \quad \text{for}\  t \in (t_k, t_{k+1}].
\end{align*}
\begin{remark}
It follows from the above definitions that $\varphi^\delta$, $\mathbf{u}^\delta_c$ are continuous piecewise linear functions in time, while
 $\widehat{\ub}_c^\delta$, $\widehat{P}_m^\delta$, $\widehat{\varphi}^\delta$, $\widehat{\mu}^\delta$  are piecewise constant (in time) functions being right continuous at the nodes
$\{t_{k+1}\}$ and  $\widetilde{\varphi}^\delta$ is left continuous at the nodes $\{t_k\}$.
\end{remark}

Using the above definition of approximate solutions, one can derive from the discrete problem \eqref{app1}--\eqref{app3} that the following identities hold:
\begin{eqnarray}
&& \varpi\int_0^T \left(\partial_t \ub_c^\delta ,\mathbf{v}_c\right)_c dt+2\int_0^T \left(\nu(\widetilde{\vp}_c^\delta)\mathbb{D}(\widehat{\ub}_c^\delta),\mathbb{D}(\mathbf{v}_c)\right)_cdt \non\\
&&+\int_0^T\left(\frac{\Pi}{\nu(\widetilde{\varphi}_m^\delta)}\left(\nabla \widehat{P}_m^\delta-\widehat{\mu}_m^\delta \nabla \widehat{\varphi}_m^\delta\right),\nabla q_m\right)_mdt
\non\\
&& +\sum_{i=1}^{d-1}\frac{\alpha_{BJSJ}}{\sqrt{{\rm trace} (\Pi)}} \int_0^T \int_{\Gamma_{cm}}\nu(\widetilde{\vp}_m^\delta)(\widehat{\ub}_c^\delta \cdot\btau_i)(\mathbf{v}_c\cdot\btau_i)dSdt \non\\
&& + \int_0^T
\int_{\Gamma_{cm}} \widehat{P}_m^\delta (\mathbf{v}_c\cdot \mathbf{n}_{cm}) dS dt-\int_0^T \int_{\Gamma_{cm}} (\widehat{\ub}_c^\delta \cdot \mathbf{n}_{cm})q_m dSdt  \non\\
&=& \int_0^T (\widehat{\mu}_c^\delta \nabla \widehat{\varphi}_c^\delta,\mathbf{v}_c)_c dt,\label{app1z}
 \end{eqnarray}
\begin{equation}
\int_0^T \left(\partial_t \varphi^\delta,\phi\right)dt  - \int_0^T(\widehat{\ub}^\delta  \widehat{\vp}^\delta , \nabla\phi)dt =- \int_0^T({\rm M}(\widetilde{\vp}^\delta)\nabla \widehat{\mu}^\delta,\nabla\phi) dt,\label{app2az}
\end{equation}
\begin{equation}
\int_0^T(\widehat{\mu}^\delta, \phi)dt =\frac{1}{\epsilon}\int_0^T\left(\widetilde{f}(\widehat{\varphi}^\delta,\widetilde{\varphi}^\delta),\phi\right) dt +\epsilon\int_0^T(\nabla\widehat{\varphi}^\delta,\nabla \phi)dt, \label{app2bz}
\end{equation}
\be
\int_0^T(\widehat{\ub}_m^\delta, \mathbf{v}_m)_m  dt =-\int_0^T\left(\frac{\Pi}{\nu(\widetilde{\varphi}_m^\delta)}\Big(\nabla \widehat{P}_m^{\delta}-\widehat{\mu}_m^{\delta}\nabla\widehat{\varphi}_m^{\delta}\Big),\mathbf{v}_m\right)_m dt.\label{app3z}
 \ee
for any $\mathbf{v}_c\in C_0^\infty([0,T]; \mathbf{H}_{c,{\rm div}})$, $q_m \in C^\infty([0,T]; X_m)$, $\phi\in C_0^\infty([0,T];H^1(\Omega))$ and $\mathbf{v}_m \in C^\infty([0,T]; \mathbf{L}^2(\Omega_m))$.

Besides, let $\mathcal{E}^\delta(t)$ be the piecewise linear interpolation of the discrete energy $\mathcal{E}(\ub_c^{k}, \vp^k)$ at $t_k$ such that
\be \mathcal{E}^\delta(t)=\frac{t_{k+1}-t}{\delta}\mathcal{E}(\ub_c^{k}, \vp^k)+\frac{t-t_k}{\delta}\mathcal{E}(\ub_c^{k+1}, \vp^{k+1}), \quad \text{for}\ t \in [t_k, t_{k+1}],
\ee
and $\mathcal{D}^\delta(t)$ be the interpolated approximate dissipation
\bea
\mathcal{D}^\delta(t)&=& 2 \left(\nu(\vp_c^k)\mathbb{D}(\ub_c^{k+1}), \mathbb{D}(\ub_c^{k+1})\right)_c+ \left( \nu(\varphi_m^k)\Pi^{-1}\ub_m^{k+1}, \ub_m^{k+1}\right)_m\non\\
 && +\int_\Omega \mathrm{M}(\vp^k)|\nabla\mu^{k+1}|^2 dx\non\\
 && +\frac{\alpha_{BJSJ}}{\sqrt{{\rm trace}(\Pi)}}\sum_{i=1}^{d-1}\int_{\Gamma_{cm}} \nu(\vp^k_m)|\ub_c^{k+1}\cdot\btau_i|^2 dS, \quad \text{for}\ t \in (t_k, t_{k+1}),
  \non
\eea
Then it follows from the discrete energy estimate \eqref{DisEnLaw} that for $k=0,1, \cdots, N-1$
\be
\frac{d}{dt}\mathcal{E}^\delta(t)=\frac{1}{\delta}(\mathcal{E}(\ub_c^{k+1}, \vp^{k+1})-\mathcal{E}(\ub_c^{k}, \vp^k))\leq - \mathcal{D}^\delta(t),\quad \text{for}\ t \in (t_k, t_{k+1}).\label{bela}
\ee
In particular, we have for $t\in [0,T]$,
\be
\mathcal{E}(\widehat{\ub}_c^{\delta}(t), \widehat{\vp}^{\delta}(t))+ \int_0^t \mathcal{D}^\delta(t) dt \leq \mathcal{E}(\ub_0, \vp_0),\quad \forall\, t\in [0,T].\label{ee1}
\ee

\subsection{Proof of Theorem \ref{thmEx}}
We now proceed to prove our main result Theorem \ref{thmEx} on the existence of finite energy weak solutions to system \eqref{weak1}--\eqref{weak4}. To this end, we shall distinguish the two cases such that $\varpi>0$ and $\varpi=0$.
\subsubsection{Case $\varpi>0$}
In order to pass to the limit as $\delta\to 0$, we first derive some \textit{a priori} estimates on the approximate solutions that are uniform in $\delta$. First, recall the mass-conservation from Lemma \ref{LMC} $$\int_\Omega (\vp^{k+1}-\vp^{k} )dx= 0, \quad \text{for}\  k=0,...,N-1,$$
which immediately yields
\be
\int_\Omega \vp^\delta dx=\int_\Omega \widehat{\vp}^\delta dx=\int_\Omega \widetilde{\vp}^\delta dx=\int_\Omega \vp_0 dx.\non
\ee
Besides, it follows from the energy estimate \eqref{ee1} that
\begin{align}
&\varpi\|\widehat{\mathbf{u}}^\delta_c\|_{L^\infty(0,T;\mathbf{L}^2(\Omega_c))}+\|\widehat{\varphi}^\delta\|_{L^\infty(0,T;H^1(\Omega))} \leq C,\label{con1} \\
&\|\mathbb{D}(\widehat{\mathbf{u}}^\delta_c)\|_{L^2(0,T;\mathbf{L}^2(\Omega_c))}+\sum_{i=1}^{d-1}\|\widehat{\ub}_c^\delta\cdot\btau_i\|_{L^2(0,T; L^2(\Gamma_{cm}))}\leq C, \\
&\|\widehat{\mathbf{u}}_m^\delta\|_{L^2(0,T;\mathbf{L}^2(\Omega_m))}\leq C, \\
&\|\nabla \widehat{\mu}^\delta\|_{L^2(0,T;\mathbf{L}^2(\Omega))} \leq C,\label{mues1}
\end{align}
where the constant $C$ depends on $\mathcal{E}(\ub_0, \vp_0)$ and $\Omega$ but is independent of $\delta$.
Taking $\phi=1$ in \eqref{app2b}, we have for $k=0,1,...,N-1$
\be
\left|\int_\Omega \mu^{k+1} dx\right|\leq \epsilon^{-1}\int_\Omega (|\vp^{k+1}|^3+|\vp^k|) dx\leq C,\non
\ee
which combined with the Poincar\'e inequality and \eqref{mues1} implies that
$$\|\widehat{\mu}^\delta\|_{L^2(0,T;H^1(\Omega))} \leq C_T,$$
where the constant $C_T$ depends on $\mathcal{E}(\ub_0, \vp_0)$, $\Omega$ and $T$. Then similar to the Neumann problem \eqref{p2ndStr}, we can apply  the elliptic estimate (similar to \eqref{p2ndLapH3}) to get
\be
\|\widehat{\varphi}^\delta\|_{L^2(0,T;H^3(\Omega))} \leq C_T.\label{esphil2h3}
\ee
Using \eqref{app3}, the above estimates, the H\"{o}lder inequality and the Gagliardo-Nirenberg inequality, we can obtain the following estimates for $\widehat{P}_m$ such that when $d=3$
\bea
&& \int_0^T \|\nabla \widehat{P}_m^\delta\|_{\mathbf{L}^2(\Omega_m)}^\frac{8}{5} dt\non\\
& \leq& C\int_0^T \Big(\|\widehat{\mathbf{u}}_m^\delta\|_{\mathbf{L}^2(\Omega_m)}^\frac{8}{5} + \|\nabla \widehat{\vp}_m^\delta\|_{\mathbf{L}^3(\Omega_m)}^\frac{8}{5}\| \widehat{\mu}_m^\delta\|_{L^6(\Omega_m)}^\frac{8}{5}\Big) dt\non\\
&\leq& C\int_0^T (\|\widehat{\mathbf{u}}_m^\delta\|_{\mathbf{L}^2(\Omega_m)}^2+1)dt+ C \sup_{0\leq t\leq T} \|\widehat{\vp}_m^\delta\|_{H^1(\Omega_m)}^\frac65 \int_0^T  \|\widehat{\vp}_m^\delta\|_{H^3(\Omega_m)}^\frac{2}{5}\| \widehat{\mu}_m^\delta\|_{H^1(\Omega_m)}^\frac{8}{5} dt\non\\
&\leq& C\int_0^T (\|\widehat{\mathbf{u}}_m^\delta\|_{\mathbf{L}^2(\Omega_m)}^2+1)dt
\non\\
&& + C \sup_{0\leq t\leq T} \|\widehat{\vp}_m^\delta\|_{H^1(\Omega_m)}^\frac{6}{5} \Big(\int_0^T  \|\widehat{\vp}_m^\delta\|_{H^3(\Omega_m)}^2 dt\Big)^{\frac{1}{5}}\Big(\int_0^T \| \widehat{\mu}_m^\delta\|_{H^1(\Omega_m)}^2 dt\Big)^{\frac{4}{5}}\non\\
&\leq& C_T, \label{pm}
\eea
and when $d=2$
\bea
&& \int_0^T \|\nabla \widehat{P}_m^\delta\|_{\mathbf{L}^2(\Omega_m)}^\frac{2r}{1+r} dt\non\\
& \leq& C\int_0^T \Big(\|\widehat{\mathbf{u}}_m^\delta\|_{\mathbf{L}^2(\Omega_m)}^\frac{2r}{1+r} + \|\nabla \widehat{\vp}_m^\delta\|_{\mathbf{L}^\frac{2r}{r-2}(\Omega_m)}^\frac{2r}{1+r}\| \widehat{\mu}_m^\delta\|_{L^r(\Omega_m)}^\frac{2r}{1+r}\Big) dt\non\\
&\leq&
 C\int_0^T (\|\widehat{\mathbf{u}}_m^\delta\|_{\mathbf{L}^2(\Omega_m)}^2+1)dt+
 C \sup_{0\leq t\leq T} \|\widehat{\vp}_m^\delta\|_{H^1(\Omega_m)}^\frac{2(r-1)}{1+r}  \int_0^T  \|\widehat{\vp}_m^\delta\|_{H^3(\Omega_m)}^\frac{2}{1+r}\| \widehat{\mu}_m^\delta\|_{H^1(\Omega_m)}^\frac{2r}{1+r} dt\non\\
&\leq& C\int_0^T (\|\widehat{\mathbf{u}}_m^\delta\|_{\mathbf{L}^2(\Omega_m)}^2+1)dt
\non\\
&& + C \sup_{0\leq t\leq T} \|\widehat{\vp}_m^\delta\|_{H^1(\Omega_m)}^\frac{2(r-1)}{1+r}  \Big(\int_0^T  \|\widehat{\vp}_m^\delta\|_{H^3(\Omega_m)}^2 dt\Big)^{\frac{1}{1+r}}\Big(\int_0^T \| \widehat{\mu}_m^\delta\|_{H^1(\Omega_m)}^2 dt\Big)^{\frac{r}{1+r}}\non\\
&\leq& C_T, \quad \text{for any}\ r\in (2, +\infty).\label{pmd2}
\eea

Based on the above estimates \eqref{con1}--\eqref{pmd2} which are independent of $\delta$, we can see that there exists a subsequence $\{(\widehat{\ub}_c^\delta, \widehat{P}_m^\delta, \widehat{\vp}^\delta, \widehat{\mu}^\delta)\}$ (still denoted by the same symbols for simplicity) as $\delta \to 0$ (or equivalently $N\to +\infty$) such that
\be
\begin{cases}
\widehat{\mathbf{u}}_c^\delta \rightarrow \mathbf{u}_c & \text{ weakly star in } L^\infty(0,T; \mathbf{L}^2(\Omega_c)),\\
&\text{ weakly  in } L^2(0,T; \mathbf{H}^1(\Omega_c)),\\
\widehat{P}_m \to P_m &\text{ weakly  in } L^\alpha(0,T; X_m),\\
 \widehat{\mathbf{u}}_m^\delta \rightarrow  \mathbf{u}_m &\text{ weakly  in } L^2(0,T; \mathbf{L}^2(\Omega_m)),\\
\widehat{\varphi}^\delta \rightarrow \varphi & \text{ weakly star in } L^\infty(0,T; H^1(\Omega)), \\
& \text{ weakly  in } L^2(0,T; H^3(\Omega)), \\
\widehat{\mu}^\delta \rightarrow \mu &\text{ weakly in } L^2(0,T; H^1(\Omega)),
\end{cases}
\label{conwmu}
\ee
for certain functions $(\ub_c, P_m, \ub_m, \vp, \mu)$ satisfying
\begin{align*}
& \mathbf{u}_c \in L^\infty(0, T; \mathbf{L}^2(\Omega_c))\cap L^2(0, T; \mathbf{H}^1(\Omega_c)), \\
& P_m \in L^\alpha(0, T; X_m),  \\
& \ub_m \in L^2(0, T; \mathbf{L}^2(\Omega_m)),\\
& \vp\in L^\infty(0, T; H^1(\Omega))\cap L^2(0, T; H^3(\Omega)),\\
& \mu \in L^2(0,T; H^1(\Omega)),
\end{align*}
where $\alpha=\frac85$ when $d=3$ and $\alpha \in (\frac43,2)$ that can be arbitrary close to $2$ when $d=2$.

In order to pass to the limit in nonlinear terms, we need to show the strong convergence of  $\widehat{\vp}^\delta$ (up to a subsequence). It follows from  equation \eqref{app2az}, the Gagliardo-Nirenberg inequality and the Sobolev embedding theorem that
\bea
&& \|\partial_t \varphi^\delta\|^{\frac{8}{5}}_{L^{\frac{8}{5}}(0,T;(H^1(\Omega))^\prime)}\non\\
& \leq& C\int_0^T \left(\|\nabla \widehat{\mu}^\delta\|_{\mathbf{L}^2(\Omega)}^\frac{8}{5}
      + \|\widehat{\vp}^\delta\|_{L^\infty(\Omega)}^\frac{8}{5}\|\widehat{\ub}^\delta\|_{\mathbf{L}^2(\Omega)}^\frac{8}{5}\right) dt\non\\
&\leq& C\int_0^T \|\nabla \widehat{\mu}^\delta\|_{\mathbf{L}^2(\Omega)}^{\frac{8}{5}}dt
   + C \sup_{0\leq t\leq T} \|\widehat{\vp}^\delta\|_{L^6(\Omega)}^\frac{6}{5}
   \int_0^T  \|\widehat{\vp}^\delta\|_{H^3(\Omega)}^\frac{2}{5}\| \widehat{\ub}^\delta\|_{\mathbf{L}^2(\Omega)}^\frac{8}{5} dt\non\\
&\leq& C\int_0^T \Big(\|\nabla \widehat{\mu}^\delta\|_{\mathbf{L}^2(\Omega)}^2+1\Big)dt+ C \sup_{0\leq t\leq T} \|\widehat{\vp}^\delta\|_{H^1(\Omega)}^\frac{6}{5} \int_0^T  \Big(\|\widehat{\vp}^\delta\|_{H^3(\Omega)}^2+\| \widehat{\ub}^\delta\|_{\mathbf{L}^2(\Omega)}^2\Big) dt\non\\
&\leq& C_T, \quad \text{when} \ d=3.\label{vpt1}
\eea
For $d=2$, we use the Br\'{e}zis-Gallouet interpolation inequality (cf. \cite{BG})
$$ \|g\|_{L^\infty(\Omega)}\leq C\|g\|_{H^1(\Omega)}\sqrt{\ln (1+\|g\|_{H^2(\Omega)})} +C(1+\|g\|_{H^1(\Omega)}), \quad \forall\, g\in H^2(\Omega)$$
to obtain that for any $\alpha\in (1,2)$, it holds
\bea
&& \|\partial_t \varphi^\delta\|^{\alpha}_{L^{\alpha}(0,T;(H^1(\Omega))^\prime)}\non\\
& \leq& C\int_0^T \left(\|\nabla \widehat{\mu}^\delta\|_{\mathbf{L}^2(\Omega)}^\alpha
      + \|\widehat{\vp}^\delta\|_{L^\infty(\Omega)}^\alpha\|\widehat{\ub}^\delta\|_{\mathbf{L}^2(\Omega)}^\alpha\right) dt\non\\
&\leq& C\int_0^T \|\nabla \widehat{\mu}^\delta\|_{\mathbf{L}^2(\Omega)}^{\alpha}dt\non\\
&&
   + C(1+ \sup_{0\leq t\leq T} \|\widehat{\vp}^\delta\|_{H^1(\Omega)}^\alpha)
   \int_0^T  \Big(1+\sqrt{\ln (1+\|\vp\|_{H^2(\Omega)})} \Big)^\alpha\| \widehat{\ub}^\delta\|_{\mathbf{L}^2(\Omega)}^\alpha dt\non\\
&\leq& C\int_0^T \Big(\|\nabla \widehat{\mu}^\delta\|_{\mathbf{L}^2(\Omega)}^2+1\Big)dt\non\\
&& + C \int_0^T  \Big[\Big(1+\sqrt{\ln (1+\|\vp\|_{H^2(\Omega)})} \Big)^\frac{2\alpha}{2-\alpha}+\| \widehat{\ub}^\delta\|_{\mathbf{L}^2(\Omega)}^2\Big] dt\non\\
&\leq& C_T, \quad \text{when} \ d=2.\label{vpt1d2}
\eea
As a result, it follows that
\be
\partial_t \vp^\delta \to \partial_t \vp \quad \text{ weakly  in } L^{\alpha}(0,T; (H^1(\Omega))^\prime).\non
\ee
where $\alpha=\frac85$ when $d=3$ and $\alpha \in (1,2)$ that can be arbitrary close to $2$ when $d=2$.

 Since
$$\|\widehat{\vp}^\delta-\vp^\delta\|_{(H^1)'}=\left\|(t_{k+1}-t)\frac{(\vp^{k+1}-\vp^k)}{\delta}\right\|_{(H^1)'}\leq \delta \|\partial_t \vp^\delta\|_{(H^1)'}, \quad t\in (t_k, t_{k+1}],$$
for $k=0,1,...,N-1$, we have
\be
\int_0^T \|\widehat{\vp}^\delta-\vp^\delta\|_{(H^1)'} ^\alpha dt\leq \delta^\alpha \int_0^T \|\partial_t \vp^\delta\|_{(H^1)'}^\alpha dt\to 0, \quad \text{as}\ \delta \to 0,
\ee
which implies
\begin{align*}
\widehat{\vp}^\delta-\vp^\delta \to 0, \quad \text{strongly in}\ \ L^\alpha(0,T; (H^1)'), \text{ as } \delta \rightarrow 0.
\end{align*}
Similarly, one can show $\|\widetilde{\varphi}^\delta -\widehat{\varphi}^\delta\|_{L^\alpha(0,T; (H^1)')} \rightarrow 0$, as $\delta \rightarrow 0$. Thus, the sequences $\{\varphi^\delta\}$, $\{\widehat{\varphi}^\delta\}$ and $\{\widetilde{\phi}^\delta\}$, if convergent, should converge to the same limit. On the other hand, by the definition of $\vp^\delta$, it satisfies the estimates similar to \eqref{con1}, \eqref{esphil2h3} for $\widehat{\vp}^\delta$. Hence, applying the Simon's compactness lemma (cf. e.g., \cite{Si85}), we deduce that there exists $\vp^*\in L^2(0,T; H^{3-\beta}(\Omega))\cap C([0,T]; H^{1-\beta}(\Omega))$,  for a suitable subsequence,
$$ \vp^\delta \to \vp^*,  \quad \text{strongly in}\ L^2(0,T; H^\beta(\Omega)),\quad
\text{as}\ \delta\to 0,
$$
for certain $0< \beta \leq 1$. In particular, we have $\vp^*=\vp$ and up to a subsequence,
\be
\widehat{\vp}^\delta, \widetilde{\vp}^\delta \to \vp \quad \text{strongly in}\ \ L^2(0,T;  H^{3-\beta}(\Omega))\cap C([0,T]; H^{1-\beta}(\Omega)), \text{ as } \delta \rightarrow 0.\label{vpt}
\ee
Concerning the initial data, since by definition $\vp^\delta|_{t=0}=\vp_0$, we infer from \eqref{vpt}  that
$$\vp|_{t=0}= \vp_0.$$
 Indeed, by \eqref{con1}, \eqref{vpt1} and \cite[Lemma 4.1]{Ab09}, we also have $\vp\in C_w([0,T]; H^1(\Omega))$.

Using similar arguments for \eqref{pm} and \eqref{pmd2}, we can deduce from \eqref{app1z} and \eqref{pm} that (taking $q_m=0$)
\bea
 && \|\partial_t \ub_c^\delta\|^{\frac{8}{5}}_{L^{\frac{8}{5}}(0,T;(\mathbf{H}^1(\Omega))^\prime)}\non\\
&\leq& C\int_0^T \left(\|\widehat{\ub}_c^\delta\|_{\mathbf{H}^1(\Omega_c)}^\frac85+ \|\widehat{P}_m^\delta \|_{H^1(\Omega_m)}^\frac85+\|\widehat{\mu}_c^\delta\|_{L^6(\Omega_c)}^\frac85\| \nabla \widehat{\varphi}_c^\delta\|_{\mathbf{L}^3(\Omega_c)}^\frac85\right) dt\non\\
&\leq& C_T, \quad \text{when}\ d=3\label{ut}
\eea
and
\bea
 && \|\partial_t \ub_c^\delta\|^{\frac{2r}{1+r}}_{L^{\frac{2r}{1+r}}(0,T;(\mathbf{H}^1(\Omega))^\prime)}\non\\
&\leq& C\int_0^T \left(\|\widehat{\ub}_c^\delta\|_{\mathbf{H}^1(\Omega_c)}^{\frac{2r}{1+r}}+ \|\widehat{P}_m^\delta \|_{H^1(\Omega_m)}^{\frac{2r}{1+r}}+\|\widehat{\mu}_c^\delta\|_{L^r(\Omega_c)}^{\frac{2r}{1+r}}\| \nabla \widehat{\varphi}_c^\delta\|_{\mathbf{L}^\frac{2r}{r-2}(\Omega_c)}^{\frac{2r}{1+r}}\right) dt\non\\
&\leq& C_T, \quad \forall\, r\in (2+\infty), \ \text{when}\ d=2\label{ut2}
\eea
Parallel to the arguments for  $\vp^\delta$, $\widehat{\vp}^\delta$, the above estimates yield that as $\delta \rightarrow 0$,
\begin{align}
 \widehat{\ub}_c^\delta-\ub_c^\delta \to 0, \quad &\text{strongly in}\ \ L^\alpha(0,T; (\mathbf{H}^1(\Omega_c))'),  \\
 \widehat{\ub}_c^\delta, \ub_c^\delta \to \ub_c, \quad &\text{strongly in}\ \ L^2(0,T; \mathbf{H}^\beta(\Omega_c))\cap C([0,T]; \mathbf{H}^{-\beta}(\Omega_c)), \label{ubt}
\end{align}
for some $\beta\in (0,1)$, $\alpha=\frac85$ when $d=3$ and $\alpha \in (\frac43,2)$ that can be arbitrary close to $2$ when $d=2$.
 Moreover, we have $\ub_c|_{t=0}= \ub_0$ and $\ub_c\in C_w([0,T]; \mathbf{L}^2(\Omega_c))$.

Based on the strong convergence \eqref{vpt} and the Sobolev embedding theorem, we can derive that
\be
\widetilde{f}(\widehat{\varphi}^\delta,\widetilde{\varphi}^\delta) \to f(\vp),  \quad \text{strongly in}\ \ L^2(0, T; L^2(\Omega)).
\ee
By the assumptions (A1)--(A2), we get
\begin{align*}
\nu(\widetilde{\vp}^\delta)\to \nu(\vp), \quad &\text{strongly in}\ \ C([0,T]; H^{1-\beta}(\Omega)),\\
{\rm M}(\widetilde{\vp}^\delta)\to {\rm M}(\vp),\quad & \text{strongly in}\ \ C([0,T]; H^{1-\beta}(\Omega)).
\end{align*}
 Similar to the argument in \eqref{pm}, we have $\widehat{\mu}^\delta\nabla \widehat{\vp}^\delta\in L^\alpha(0,T; \mathbf{L}^2(\Omega))$ with $\alpha$ being the parameter specified above. Moreover, we infer from the strong convergence of $\widehat{\vp}^\delta$ (see \eqref{vpt}) and the weak convergence of $\widehat{\mu}^\delta$ (see \eqref{conwmu}) that
$$\widehat{\mu}^\delta \nabla \widehat{\varphi}^\delta\to \mu\nabla \vp$$ in the distribution sense. At last, we note that in \eqref{app1z}--\eqref{app2az}, after integration by parts, we get
\bea
&&\int_0^T\left(\partial_t \ub_c^\delta ,\mathbf{v}_c\right)_c dt =- \int_0^T\left( \ub_c^\delta ,\partial_t \mathbf{v}_c\right)_c dt,\non\\
&& \int_0^T \left(\partial_t \varphi^\delta,\phi\right)dt =-\int_0^T \left( \varphi^\delta,\partial_t \phi\right)dt.\non
\eea
 Using the above convergence results, we are able to pass to the limit in Eqs. \eqref{app1z}--\eqref{app3z} to see that the limit functions $(\ub_c, P_m, \ub_m, \vp, \mu)$ satisfy the weak formulation \eqref{weak1}--\eqref{weak4} (see Definition \ref{defweak}).

Finally, we show that $(\ub_c, \ub_m, \vp, \mu)$ also fulfills the energy inequality \eqref{Energyinq}. The energy estimate \eqref{bela} yields that
\be
\mathcal{E}(\ub_{0}, \vp_{0})h(0)+\int_0^T \mathcal{E}^\delta(t) h'(t) dt \geq \int_0^T \mathcal{D}^\delta(t) h(t) dt,\label{bel2}
\ee
for all $h(t)\in W^{1,1}(0,T)$ with $h\geq 0$ and $h(T)=0$. On the other hand, it follows from the strong convergence results \eqref{vpt} and \eqref{ubt} that as $\delta \to 0$, for almost every $t\in (0, T)$, we have (up to a subsequence),
\bea
&& \widehat{\ub}_c^\delta(t)\to \ub_c(t), \quad \text{strongly in}\ \ \mathbf{L}^2(\Omega_c),\non\\
&& \widehat{\vp}^\delta(t) \to \vp(t), \quad\ \text{strongly in}\ \ H^1(\Omega),\non
\eea
which imply that
\be
\mathcal{E}^\delta(t)\to \mathcal{E}(\ub_c(t), \vp(t)),\quad \text{for almost every} \ t\in (0, T).\non
\ee
By the lower semi-continuity of norms and the almost everywhere convergence of $\nu(\widetilde{\vp}^\delta)$, $\mathrm{M}(\widetilde{\vp}^\delta)$, we have
\be
\liminf_{\delta\to 0} \int_s^t \mathcal{D}^\delta(\tau) h(\tau)  d\tau \geq \int_s^t \mathcal{D}(\tau) h(\tau)  d\tau,\quad \text{for}\ 0\leq s<t\leq T,\non
\ee
where $\mathcal{D}(t)$ is defined as in \eqref{D}. Passing to the limit in \eqref{bel2}, we get
\be
\mathcal{E}(\ub_{0}, \vp_{0})h(0)+\int_0^T \mathcal{E}(\ub_c(t), \vp(t)) h'(t) dt \geq \int_0^T \mathcal{D}(t) h(t)  dt.\non
\ee
Then we can conclude from \cite[Lemma 4.3]{Ab09} that the energy inequality \eqref{Energyinq} holds for all $s\leq t<T$ and almost all $0\leq s<T$ including $s=0$.

\subsubsection{Case $\varpi=0$}
If $\varpi=0$, one does not have a direct estimate on $\|\widehat{\mathbf{u}}^\delta_c\|_{\mathbf{L}^2(\Omega_c)}$ (compare to \eqref{con1}). Recall also  that in our domain setting, the boundary  $\Gamma_c=\emptyset$ is allowed, i.e., $\Omega_c$ can be enclosed completely by $\Omega_m$. As a consequence, the classical Korn's inequality \eqref{kornsimp} does not apply. To overcome this difficulty, we shall derive an equivalent norm on the following space
$$
 \mathbf{Z}=\{\mathbf{u}\ |\ \ub_c=\ub|_{\Omega_c}\in \mathbf{H}_{c, \mathrm{div}},\ \ub_m=\ub|_{\Omega_m}\in \mathbf{H}_{m,\mathrm{div}}, \ \ub_c\cdot\mathbf{n}_{cm}=\ub_m\cdot\mathbf{n}_{cm} \ \text{on}\ \Gamma_{cm}\}.
$$

\begin{lemma}\label{equinorml}
The norm given by
\begin{align}\label{equinorm}
\|\mathbf{u}\|_{\mathbf{Z}}^2:=\|\mathbb{D}(\mathbf{u}_c)\|_{\mathbf{L}^2(\Omega_c)}^2
+\sum_{i=1}^{d-1}\|\mathbf{u}_c \cdot \mathbf{\btau}_i\|_{L^2(\Gamma_{cm})}^2+
\|\mathbf{u}_m\|_{\mathbf{L}^2(\Omega_m)}^2.
\end{align}
 is an equivalent norm on $\mathbf{Z}$.
\end{lemma}
\begin{proof}
The case that $\Gamma_c$ has positive measure is trivial in view of Korn's inequality \eqref{kornsimp}. Below we focus on the situation where $\Omega_m$ encloses completely $\Omega_c$.
It is clear from Korn's inequality \eqref{korn} and the trace theorem that the following quantity defines an equivalent norm on $\mathbf{Z}$
\begin{align}
|\!|\!|\mathbf{u}|\!|\!|^2:=\|\mathbb{D}(\mathbf{u}_c)\|_{\mathbf{L}^2(\Omega_c)}^2
+\|\mathbf{u}_c\|_{\mathbf{L}^2(\Omega_c)}^2+\sum_{i=1}^{d-1}\|\mathbf{u}_c \cdot \btau_i\|_{L^2(\Gamma_{cm})}^2+\|\mathbf{u}_m\|_{\mathbf{L}^2(\Omega_m)}^2.\label{equinormb}
\end{align}
One only needs to prove there exists a constant $C$ independent of the function $\ub$ such that
\begin{align*}
|\!|\!|\mathbf{u}|\!|\!| \leq C \|\mathbf{u}\|_{\mathbf{Z}},\quad  \forall\, \mathbf{u} \in \mathbf{Z}.
\end{align*}
Suppose by contradiction that for a sequence $\{\mathbf{u}_n\} $ in $\mathbf{Z}$ it holds
\begin{align}
|\!|\!|\mathbf{u}_n|\!|\!| \geq n \|\mathbf{u}_n\|_{\mathbf{Z}}.\label{nnn}
\end{align}
By homogeneity, we may normalize $|\!|\!|\mathbf{u}_n|\!|\!|=1$. Then $\{\mathbf{u}_n\}$ is a bounded sequence in $\mathbf{Z}$. There exists a subsequence, still denoted by $\{\mathbf{u}_n\}$, such that $\mathbf{u}_n $ converges to $\mathbf{u}_\infty$ weakly in $\mathbf{Z}$. In particular, one has by Sobolev compact embedding $\mathbf{u}_{c_n}:=\ub_n|_{\Omega_c}$ converges to $\mathbf{u}_{c_\infty}$ strongly in $\mathbf{L}^2(\Omega_c)$. On the other hand, due to \eqref{nnn},
 \be
\|\mathbf{u}_n\|_{\mathbf{Z}}\to 0.\label{condiv}
 \ee
 It follows from the definitions \eqref{equinorm} and \eqref{equinormb} that $\|\mathbf{u}_{c_n}\|_{\mathbf{L}^2(\Omega_c)}$ converges to $1$, which implies
\begin{align}\label{temp1}
\|\mathbf{u}_{c_\infty}\|_{\mathbf{L}^2(\Omega_c)}=1.
\end{align}
Using the facts that $\mathbf{u}_{m_n}:=\ub_n|_{\Omega_m} \in \mathbf{H}_{m, \mathrm{div}}$, \eqref{condiv} and the trace theorem, we see that
\begin{align*}
\mathbf{u}_{m_n} \cdot \mathbf{n}_{cm}\big|_{\Gamma_{cm}} \rightarrow 0,\quad  \text{ in } H^{-\frac{1}{2}}(\Gamma_{cm}).
\end{align*}
Since $\mathbf{u}_n \in \mathbf{Z}$, by the continuity condition on the interface $\Gamma_{cm}$, one concludes
\begin{align*}
\mathbf{u}_{c_n} \cdot \mathbf{n}_{cm}\big|_{\Gamma_{cm}}=\mathbf{u}_{m_n} \cdot \mathbf{n}_{cm}\big|_{\Gamma_{cm}} \rightarrow 0, \quad \text{ in } H^{\frac{1}{2}}(\Gamma_{cm}).
\end{align*}
On the other hand, \eqref{condiv} implies $\|\mathbf{u}_{c_n}\cdot \btau_i\|_{L^2(\Gamma_{cm})} \rightarrow 0$ $(i=1,...,d-1)$. As a consequence of the above estimates and the fact that $\ub_{c_\infty}$ is the weak limit of $\ub_{c_n}$ in $\mathbf{H}^1(\Omega_c)$, we obtain
\begin{align}\label{temp2}
\mathbf{u}_{c_\infty}\big|_{\Gamma_{cm}}=\mathbf{0}.
\end{align}
Finally, by the weak lower semi-continuity of norm, one has
\begin{align}\label{temp3}
\|\mathbb{D}(\mathbf{u}_{c_\infty})\|_{\mathbf{L}^2(\Omega_c)} \leq \liminf_{n\rightarrow \infty} \|\mathbb{D}(\mathbf{u}_{c_n})\|_{\mathbf{L}^2(\Omega_c)}=0.
\end{align}
By virtue of \eqref{temp2} and \eqref{temp3}, we infer from the Korn's inequality \eqref{kornsimp} that
\begin{align*}
\|\mathbf{u}_{c_\infty}\|_{\mathbf{L}^2(\Omega_c)}=0.
\end{align*}
This leads to a contradiction with \eqref{temp1}. The proof is complete.
\end{proof}
Now we return to the proof of Theorem \ref{thmEx}. It follows easily from Lemma \ref{LMC} and the definition of $\widehat{\mathbf{u}}_c^\delta, \widehat{\mathbf{u}}_m^\delta$ that
\begin{align*}
\widehat{\mathbf{u}}_m^\delta \in \mathbf{H}_{m, \mathrm{div}}, \quad \widehat{\mathbf{u}}_m^\delta \cdot \mathbf{n}_{cm}= \widehat{\mathbf{u}}_c^\delta \cdot \mathbf{n}_{cm} \text{ in } H^{\frac{1}{2}}(\Gamma_{cm}).
\end{align*}
Thus, the equivalent norm \eqref{equinorm} in Lemma \ref{equinorml} is applicable, and one can derive estimate on $\|\widehat{\mathbf{u}}_c^{\delta}\|_{L^2(0,T; \mathbf{H}^1(\Omega_c))}$ from the energy estimate \eqref{ee1}. Then one can conclude the proof as in the case of $\varpi >0$.

The proof of Theorem \ref{thmEx} is complete.


\section{Weak-strong Uniqueness}\setcounter{equation}{0}

In this section, we prove the uniqueness result Theorem \ref{thmuni}.
Below we just give the proof for $d=3$, while the proof for $d=2$ can be obtained with minor modifications under certain weaker regularity assumptions.

First, we recall that the finite energy weak solution $(\ub_c, P_m, \ub_m, \vp, \mu)$ to CHSD system \eqref{HSCH-NSCH1}--\eqref{IBCi7} satisfy the energy inequality \eqref{Energyinq}, i.e.,
\bea
&& \int_{\Omega_c}\frac{\varpi}{2}|\ub_c(t)|^2 dx+\int_{\Omega}\left[\frac{\epsilon}{2}|\nabla\varphi|^2+\frac{1}{\epsilon}F(\varphi)\right]dx\nonumber\\
&& + \int_0^t \int_{\Omega_m} \nu(\varphi_m)\Pi^{-1}|\ub_m|^2dx d\tau
+\int_0^t \int_{\Omega_c}2\nu(\vp_c)|\mathbb{D}(\ub_c)|^2dxd\tau
\non\\
&& +\int_0^ t\int_{\Omega}{\rm M}(\vp)|\nabla\mu(\vp)|^2dx d\tau
+\frac{\alpha_{BJSJ}}{\sqrt{{\rm trace}(\Pi)}}\sum_{i=1}^{d-1}\int_0^t \int_{\Gamma_{cm}} \nu(\vp_m)|\ub_c\cdot\btau_i|^2 dS d\tau\nonumber\\
&\leq& \int_{\Omega_c}\frac{\varpi}{2}|\ub_0|^2 dx+\int_{\Omega}\left[\frac{\epsilon}{2}|\nabla\varphi_0|^2+\frac{1}{\epsilon}F(\varphi_0)\right]dx.\label{EE1}
\eea
On the other hand, 
the regular solution $(\tilde{\ub}_c, \tilde{P}_m, \tilde{\ub}_m, \tilde{\vp}, \tilde{\mu})$ are allowed to be used as the test functions for the CHSD system and the following energy equality holds
\bea
&& \int_{\Omega_c}\frac{\varpi}{2}|\tilde{\ub}_c(t)|^2 dx+\int_{\Omega}\left[\frac{\epsilon}{2}|\nabla\tilde{\varphi}|^2
+\frac{1}{\epsilon}F(\tilde{\varphi})\right]dx\nonumber\\
&& + \int_0^t \int_{\Omega_m} \nu(\tilde{\varphi}_m)\Pi^{-1}|\tilde{\ub}_m|^2dx d\tau
+\int_0^t \int_{\Omega_c}2\nu(\tilde{\vp}_c)|\mathbb{D}(\tilde{\ub}_c)|^2dxd\tau
\non\\
&& +\int_0^ t\int_{\Omega}{\rm M}(\tilde{\vp})|\nabla \tilde{\mu}(\tilde{\vp})|^2dx d\tau
+\frac{\alpha_{BJSJ}}{\sqrt{{\rm trace}(\Pi)}}\sum_{i=1}^{d-1}\int_0^t \int_{\Gamma_{cm}} \nu(\tilde{\vp}_m)|\tilde{\ub}_c\cdot\btau_i|^2 dS d\tau\nonumber\\
&=& \int_{\Omega_c}\frac{\varpi}{2}|\ub_0|^2 dx+\int_{\Omega}\left[\frac{\epsilon}{2}|\nabla\varphi_0|^2+\frac{1}{\epsilon}F(\varphi_0)\right]dx.\label{EE2}
\eea

Next, taking $\tilde{\ub}$ and $-\epsilon \Delta \tilde{\varphi}$ as test functions in the weak formulation for the finite energy weak solution $(\ub_c, P_m, \ub_m, \vp, \mu)$ and using the equations for $\tilde{\ub}_c$, $\tilde{\vp}$, we obtain that
\bea
&& \varpi(\ub_c(t), \tilde{\ub}_c(t))_c- \varpi \int_{\Omega_c}  |\ub_0|^2 dx\nonumber\\
&=&\varpi \int_0^t (\ub_c, \partial_t \tilde{\ub}_c)_c d\tau -\frac{\alpha_{BJSJ}}{\sqrt{{\rm trace}(\Pi)}}\sum_{i=1}^{d-1}\int_0^t \int_{\Gamma_{cm}} \nu(\vp_m)(\ub_c\cdot\btau_i) (\tilde{\ub}_c\cdot\btau_i) dS d\tau \nonumber\\
 && -\int_0^t\int_{\Gamma_{cm}}P_m(\tilde{\ub}_c\cdot\mathbf{n}_{cm}) dS d\tau
  -\int_0^t \int_{\Omega_c} 2\nu(\vp_c)\mathbb{D}(\ub_c): \mathbb{D}(\tilde{\ub}_c) dx d\tau\non\\
  &&+\int_0^t (\mu_c\nabla \vp_c,  \tilde{\ub}_c)_c d\tau\non\\
  &=&-\frac{\alpha_{BJSJ}}{\sqrt{{\rm trace}(\Pi)}}\sum_{i=1}^{d-1}\int_0^t \int_{\Gamma_{cm}} \nu(\vp_m)(\ub_c\cdot\btau_i) (\tilde{\ub}_c\cdot\btau_i) dS d\tau \nonumber\\
  && -\frac{\alpha_{BJSJ}}{\sqrt{{\rm trace}(\Pi)}}\sum_{i=1}^{d-1}\int_0^t \int_{\Gamma_{cm}} \nu(\tilde{\vp}_m)(\ub_c\cdot\btau_i) (\tilde{\ub}_c\cdot\btau_i) dS d\tau \nonumber\\
 && -\int_0^t\int_{\Gamma_{cm}}P_m(\tilde{\ub}_c\cdot\mathbf{n}_{cm}) dS d\tau
 -\int_0^t\int_{\Gamma_{cm}}\tilde{P}_m(\ub_c\cdot\mathbf{n}_{cm}) dS d\tau
  \non\\
  && -\int_0^t \int_{\Omega_c} 2\nu(\vp_c)\mathbb{D}(\ub_c): \mathbb{D}(\tilde{\ub}_c) dx d\tau
  -\int_0^t \int_{\Omega_c} 2\nu(\tilde{\vp}_c)\mathbb{D}(\ub_c): \mathbb{D}(\tilde{\ub}_c) dx d\tau\non\\
  &&+\int_0^t (\mu_c\nabla \vp_c,  \tilde{\ub}_c)_c d\tau +\int_0^t (\tilde{\mu}_c\nabla \tilde{\vp}_c,  \ub_c)_c d\tau, \label{EE3}
\eea
 \bea
 && \epsilon \int_\Omega \nabla \vp(t)\cdot\nabla \tilde{\vp}(t) dx-\epsilon \int_\Omega |\nabla \vp_0|^2 dx\non\\
 &=& \epsilon \int_0^t \int_\Omega M(\vp)\nabla \mu \cdot \nabla \Delta \tilde{\vp}dx d\tau
  +\epsilon \int_0^t \int_\Omega M(\tilde{\vp})\nabla \tilde{\mu} \cdot \nabla \Delta \vp dx d\tau\non\\
  && +\epsilon\int_0^t \int_\Omega (\ub \cdot \nabla \vp) \Delta \tilde{\vp} dx d\tau
 +\epsilon\int_0^t \int_\Omega (\tilde{\ub} \cdot \nabla \tilde{\vp}) \Delta \vp dx d\tau,\label{EE4}
 \eea
 \bea
&& \int_0^t \int_{\Omega_m} \nu(\varphi_m)\Pi^{-1} \ub_m \cdot \tilde{\ub}_m dx d\tau +\int_0^t \int_{\Omega_m} \nu(\tilde{\varphi}_m)\Pi^{-1} \ub_m \cdot \tilde{\ub}_m dx d\tau \non\\
&=& -\int_0^t\left(\big(\nabla P_m-\mu_m\nabla\varphi_m\big), \tilde{\mathbf{u}}_m\right)_m d\tau-\int_0^t\left(\big(\nabla \tilde{P}_m-\tilde{\mu}_m\nabla\tilde{\varphi}_m\big), \mathbf{u}_m\right)_m d\tau\non\\
&=&\int_0^t\int_{\Gamma_{cm}}P_m(\tilde{\ub}_c\cdot\mathbf{n}_{cm}) dS d\tau + \int_0^t\int_{\Gamma_{cm}} \tilde{P}_m(\ub_c\cdot\mathbf{n}_{cm}) dS d\tau\non\\
&& +\int_0^t (\mu_m\nabla \vp_m,  \tilde{\ub}_m)_m d\tau+\int_0^t (\tilde{\mu}_m\nabla \tilde{\vp}_m,  \ub_m)_m d\tau.\label{EE5}
 \eea
Adding \eqref{EE1} with \eqref{EE2} and subtracting the sum of \eqref{EE3}--\eqref{EE5} from the resultant, by a direct computation we obtain that
\bea
&& \frac{\varpi}{2}\int_{\Omega_c}|\ub_c(t)-\tilde{\ub}_c(t)|^2dx + \frac{\epsilon}{2}\int_\Omega |\nabla \vp(t)-\nabla \tilde{\vp}(t)|^2 dx\non\\
&&+\int_0^t \int_{\Omega_c}2\nu(\vp_c)|\mathbb{D}(\ub_c)-\mathbb{D}(\tilde{\ub}_c)|^2dxd\tau\non\\
&& +\int_0^t\int_{\Omega_m}\nu(\vp_m)\Pi^{-1}|\ub_m-\tilde{\ub}_m|^2dx d\tau \non\\
&& +\epsilon^2\int_0^t \int_\Omega \mathrm{M}(\vp)|\nabla\Delta \vp -\nabla \Delta\tilde{\vp}|^2dx d\tau \non\\
&&+\frac{\alpha_{BJSJ}}{\sqrt{{\rm trace}(\Pi)}}\sum_{i=1}^{d-1}\int_0^t \int_{\Gamma_{cm}} \nu(\vp)|(\ub_c-\tilde{\ub}_c)\cdot\btau_i|^2 dS d\tau\non\\
&\leq &  -\int_0^t \int_{\Omega_c}2(\nu(\tilde{\vp}_c)-\nu(\vp_c))\mathbb{D}(\tilde{\ub}_c):(\mathbb{D}(\tilde{\ub}_c)-\mathbb{D}(\ub_c)) dxd\tau\non\\
&&-\int_0^t \int_{\Omega_m} (\nu(\tilde{\vp}_m)-\nu(\vp_m))\Pi^{-1}\tilde{\ub}_m (\tilde{\ub}_m-\ub_m) dxd\tau\non\\
&&-\epsilon^2\int_0^t\int_\Omega (\mathrm{M}(\tilde{\vp})-\mathrm{M}(\vp))\nabla \Delta \tilde{\vp}\cdot (\nabla\Delta \tilde{\vp}-\nabla \Delta \vp) dx d\tau\non\\
&&-\frac{\alpha_{BJSJ}}{\sqrt{{\rm trace}(\Pi)}}\sum_{i=1}^{d-1}\int_0^t \int_{\Gamma_{cm}} (\nu(\tilde{\vp}_m)-\nu(\vp_m))(\tilde{\ub}_c\cdot\btau_i)((\tilde{\ub}_c-\ub_c)\cdot\btau_i) dS d\tau\non\\
&& +2\int_0^t \int_\Omega (\mathrm{M}(\vp) \nabla \Delta \vp\cdot \nabla f(\vp) + \mathrm{M}(\tilde{\vp}) \nabla \Delta \tilde{\vp}\cdot \nabla f(\tilde{\vp})) dx d\tau\non\\
&&-\int_0^t\int_\Omega (\mathrm{M}(\vp) \nabla f(\vp) \cdot \nabla \Delta \tilde{\vp} +\mathrm{M}(\tilde{\vp}) \nabla f(\tilde{\vp})\cdot \nabla \Delta \vp) dx d\tau\non\\
&&-\frac{1}{\epsilon^2} \int_0^t \int_\Omega (\mathrm{M}(\vp)|\nabla f(\vp)|^2  + \mathrm{M}(\tilde{\vp})|\nabla f(\tilde{\vp})|^2) dx d\tau \non\\
&& + \frac{1}{\epsilon} \int_{\Omega} (2F(\varphi_0)-F(\vp)-F(\tilde{\varphi}))dx\non\\
&& +\epsilon\int_0^t\int_\Omega (\Delta \vp\nabla \vp \cdot \tilde{\ub} + \Delta \tilde{\vp}\nabla \tilde{\vp}\cdot \ub-\ub\cdot\nabla \vp \Delta \tilde{\vp}- \tilde{\ub}\cdot \nabla \tilde{\vp}\Delta \vp) dxd\tau\non\\
&:=& \sum_{j=1}^9 I_j,\label{EE6}
\eea
where we have used the incompressibility condition and the fact
$$\int_\Omega (\ub\cdot \nabla \vp) f(\vp) dx=\int_\Omega \ub\cdot \nabla F(\vp) dx=0.$$
Using the mass conservation property $\int_\Omega (\tilde{\vp}-\vp) dx=0$ (due to the choice of initial data),  the Poincar\'e inequality, the Sobolev embedding theorem and the Gagliardo-Nirenberg inequality, we have the following estimates for $\phi=\tilde{\vp}-\vp$
\bea
 \|\phi\|_{L^\infty(\Omega)} &\leq &C (\|\nabla \Delta \phi\|_{\mathbf{L}^2(\Omega)}^\frac14\|\phi\|_{L^6(\Omega)}^\frac34+\| \phi\|_{L^6(\Omega)})\non\\
 &\leq& C (\|\nabla \Delta \phi\|_{\mathbf{L}^2(\Omega)}^\frac14\|\nabla \phi\|_{\mathbf{L}^2(\Omega)}^\frac34+\|\nabla \phi\|_{\mathbf{L}^2(\Omega)}),
 \non
 \eea
 \be
 \|\Delta \phi\|_{L^3(\Omega)}\leq C\big(\|\nabla \Delta \phi\|_{\mathbf{L}^2(\Omega)}^\frac34\|\nabla \phi\|_{\mathbf{L}^2(\Omega)}^\frac14+\|\nabla \phi\|_{\mathbf{L}^2(\Omega)}\big),\non
 \ee
 \be
\|\nabla \phi\|_{\mathbf{L}^6(\Omega)}
  \leq  C \big(\|\nabla \Delta \phi\|_{\mathbf{L}^2(\Omega)}^\frac12\|\nabla \phi\|_{\mathbf{L}^2(\Omega)}^\frac12+\|\nabla \phi\|_{\mathbf{L}^2(\Omega)}\big).\non
\ee
Combining the above estimates with the Young inequality, we get
\bea
I_1&\leq& C\int_0^t\|\nu(\tilde{\vp}_c)-\nu(\vp_c)\|_{L^\infty(\Omega_c)}\|\mathbb{D}(\tilde{\ub}_c)\|_{\mathbf{L}^2(\Omega_c)}
\|\mathbb{D}(\tilde{\ub}_c)-\mathbb{D}(\ub_c)\|_{\mathbf{L}^2(\Omega_c)} d\tau \non\\
&\leq& C\int_0^t \|\tilde{\vp}-\vp\|_{L^\infty(\Omega)}\|\mathbb{D}(\tilde{\ub}_c)\|_{\mathbf{L}^2(\Omega_c)}
\|\mathbb{D}(\tilde{\ub}_c)-\mathbb{D}(\ub_c)\|_{\mathbf{L}^2(\Omega_c)} d\tau\non\\
&\leq& C\int_0^t(\|\nabla \Delta (\tilde{\vp}-\vp)\|_{\mathbf{L}^2(\Omega)}^\frac14\|\nabla (\tilde{\vp}-\vp)\|_{\mathbf{L}^2(\Omega)}^\frac34+\|\nabla (\tilde{\vp}-\vp)\|_{\mathbf{L}^2(\Omega)}) \non\\
&&\quad \times \|\mathbb{D}(\tilde{\ub}_c)\|_{\mathbf{L}^2(\Omega_c)}
\|\mathbb{D}(\tilde{\ub}_c)-\mathbb{D}(\ub_c)\|_{\mathbf{L}^2(\Omega_c)} d\tau\non\\
&\leq& \zeta\int_0^t \|\nabla \Delta (\tilde{\vp}-\vp)\|_{\mathbf{L}^2(\Omega)}^2 d\tau
 +\zeta \int_0^t \|\mathbb{D}(\tilde{\ub}_c)-\mathbb{D}(\ub_c)\|_{\mathbf{L}^2(\Omega_c)}^2 d\tau \non\\
 && +C \int_0^t \big(\|\mathbb{D}(\tilde{\ub}_c)\|_{\mathbf{L}^2(\Omega_c)}^\frac83+1\big)\|\nabla (\tilde{\vp}-\vp)\|_{\mathbf{L}^2(\Omega)}^2 d\tau,\non
\eea
where $\zeta>0$ is a small constant to be chosen later. In a similar manner, we have the following estimates for $I_2$, $I_3$ and $I_4$:
\bea
I_2&\leq& \zeta \int_0^t \|\nabla \Delta (\tilde{\vp}-\vp)\|_{\mathbf{L}^2(\Omega)}^2 d\tau+\zeta\int_0^t\|\tilde{\mathbf{u}}_m-\ub_m\|_{\mathbf{L}^2(\Omega_m)}^2d\tau \non\\
&& +C\int_0^t \big(\|\tilde{\ub}_m\|_{\mathbf{L}^2(\Omega_m)}^\frac83+1\big)\|\nabla (\tilde{\vp}-\vp)\|_{\mathbf{L}^2(\Omega)}^2 d\tau,\non
\eea
\bea
I_3&\leq& \zeta \int_0^t \|\nabla \Delta (\tilde{\vp}-\vp)\|_{\mathbf{L}^2(\Omega)}^2 d\tau\non\\
 && +C\int_0^t \big(\|\nabla \Delta \tilde{\vp}\|_{\mathbf{L}^2(\Omega)}^\frac83+1\big)\|\nabla (\tilde{\vp}-\vp)\|_{\mathbf{L}^2(\Omega)}^2 d\tau,\non
\eea
\bea
I_4&\leq& \zeta\int_0^t \|\nabla \Delta (\tilde{\vp}-\vp)\|_{\mathbf{L}^2(\Omega)}^2 d\tau + \zeta\sum_{i=1}^{d-1}\int_0^t \int_{\Gamma_{cm}}|(\tilde{\mathbf{u}}_c-\ub_c)\cdot\btau_i|^2dS d\tau \non \\
&& +C \int_0^t \big(\|\tilde{\ub}_c\|_{\mathbf{L}^2(\Omega_c)}^\frac83+1\big)\|\nabla (\tilde{\vp}-\vp)\|_{\mathbf{L}^2(\Omega)}^2 d\tau,\non
\eea
Since $\vp\in L^\infty(0,T; H^1(\Omega))\cap L^2(0,T; H^3(\Omega))$, by the Gagliardo-Nirenberg inequality we deduce that
\bea
\int_0^T \|f(\vp)\|_{H^1(\Omega)}^4dt
&\leq& C\int_0^T (\|\vp\|_{H^1(\Omega)}^4+ \|\vp\|_{L^6(\Omega)}^{12}+ 81\|\vp^2\nabla \vp\|_{\mathbf{L}^2(\Omega)}^4) dt\non\\
&\leq& C_T+C\int_0^t \|\vp\|_{L^\infty(\Omega)}^8\|\nabla \vp\|_{\mathbf{L}^2(\Omega)}^4 dt\non\\
&\leq& C_T+C\int_0^T (\|\nabla \Delta \vp\|_{\mathbf{L}^2(\Omega)}^\frac14\|\vp\|_{L^6(\Omega)}^\frac34+\|\vp\|_{L^6(\Omega)})^8dt\non\\
&\leq& C_T +C \int_0^T \|\nabla \Delta \vp\|_{\mathbf{L}^2(\Omega)}^2 dt\non\\
&\leq & C_T,\non
\eea
which implies $f(\vp)\in L^4(0,T; H^1(\Omega))\subset L^\frac83(0,T; H^1(\Omega))$. Thus we can take $f(\vp)=\vp^3-\vp$ as a test function in the Cahn-Hilliard equation for $\vp$. Since the nonlinear part $\vp^3$ is monotone
increasing, similar to \cite[Proposition 4.2]{CKRS}, we see that the dual product satisfies $\langle \vp_t,f(\vp)\rangle_{(H^1)', H^1}=\frac{d}{dt}\int_\Omega F(\vp)dx$ for a.e. $t\in(0,T)$. Then integrating with respect to $t$ we deduce that
\bea
&& \int_\Omega F(\vp)dx -\int_\Omega F(\vp_0)dx\non\\
&=&
\epsilon\int_0^t \int_\Omega \mathrm{M}(\vp)\nabla\Delta \vp \cdot \nabla f(\vp) dxd\tau -\frac{1}{\epsilon}\int_0^t \int_\Omega \mathrm{M}(\vp)|\nabla f(\vp)|^2 dx d\tau,\non
\eea
In a similar way, we have the same identity for the regular solution $\tilde{\vp}$
 \bea
&& \int_\Omega F(\tilde{\vp})dx -\int_\Omega F(\vp_0)dx\non\\
&=&
\epsilon\int_0^t \int_\Omega \mathrm{M}(\tilde{\vp})\nabla\Delta \tilde{\vp} \cdot \nabla f(\tilde{\vp}) dxd\tau -\frac{1}{\epsilon}\int_0^t \int_\Omega \mathrm{M}(\tilde{\vp})|\nabla f(\tilde{\vp})|^2 dx d\tau,\non
\eea
As a consequence, we obtain that
\bea
&& I_5+I_6+I_7+I_8\non\\
&=& \int_0^t \int_\Omega [-\mathrm{M}(\vp) \nabla \Delta (\tilde{\vp}-\vp)\cdot \nabla f(\vp) + \mathrm{M}(\tilde{\vp}) \nabla \Delta (\tilde{\vp}-\vp)\cdot \nabla f(\tilde{\vp})] dx d\tau\non\\
&=& \int_0^t \int_\Omega [\mathrm{M}(\tilde{\vp})-\mathrm{M}(\vp)]\nabla \Delta (\tilde{\vp}-\vp) \cdot \nabla f(\tilde{\vp}) dxd\tau\non\\
&&+ \int_0^t \int_\Omega \mathrm{M}(\vp)\nabla \Delta (\tilde{\vp}-\vp)\cdot (\nabla f(\tilde{\vp})-\nabla f(\vp)) dxd\tau\non\\
&:=& J_{1}+J_{2}.
\eea
The term $J_1$ can be estimated like $I_1$ such that
\bea
J_1&\leq& C\int_0^t\|\mathrm{M}(\tilde{\vp})-\mathrm{M}(\vp)\|_{L^\infty(\Omega)}
\|\nabla \Delta (\tilde{\vp}-\vp)\|_{\mathbf{L}^2(\Omega)}
\|\nabla f(\tilde{\vp})\|_{\mathbf{L}^2(\Omega)} d\tau \non\\
&\leq& C\int_0^t \|\tilde{\vp}-\vp\|_{L^\infty(\Omega)}\|\nabla \Delta (\tilde{\vp}-\vp)\|_{\mathbf{L}^2(\Omega)} \|\nabla f(\tilde{\vp})\|_{\mathbf{L}^2(\Omega)} d\tau\non\\
&\leq& C\int_0^t(\|\nabla \Delta (\tilde{\vp}-\vp)\|_{\mathbf{L}^2(\Omega)}^\frac14\|\nabla (\tilde{\vp}-\vp)\|_{\mathbf{L}^2(\Omega)}^\frac34+\|\nabla (\tilde{\vp}-\vp)\|_{\mathbf{L}^2(\Omega)}) \non\\
&&\quad \times
\|\nabla \Delta (\tilde{\vp}-\vp)\|_{\mathbf{L}^2(\Omega)} (\|\tilde{\vp}\|_{L^\infty}^2+1)\|\nabla \tilde{\vp}\|_{\mathbf{L}^2(\Omega)} d\tau \non\\
&\leq& \zeta\int_0^t \|\nabla \Delta (\tilde{\vp}-\vp)\|_{\mathbf{L}^2(\Omega)}^2 d\tau \non\\
&&+C \int_0^t \big(\|\tilde{\vp}\|_{L^\infty(\Omega)}^\frac{16}{3}+1\big)\|\nabla \tilde{\vp}\|_{\mathbf{L}^2(\Omega)}^\frac{8}{3}\|\nabla (\tilde{\vp}-\vp)\|_{\mathbf{L}^2(\Omega)}^2 d\tau\non \\
&& +C \int_0^t \big(\|\tilde{\vp}\|_{L^\infty(\Omega)}^4+1\big)\|\nabla \tilde{\vp}\|_{\mathbf{L}^2(\Omega)}^2  \|\nabla (\tilde{\vp}-\vp)\|_{\mathbf{L}^2(\Omega)}^2 d\tau,\non\\
 &\leq&  \zeta\int_0^t \|\nabla \Delta (\tilde{\vp}-\vp)\|_{\mathbf{L}^2(\Omega)}^2 d\tau
 +C \int_0^t \big(\|\nabla \Delta \tilde{\vp}\|_{L^2(\Omega)}^\frac{4}{3}+1\big) \|\nabla (\tilde{\vp}-\vp)\|_{\mathbf{L}^2(\Omega)}^2 d\tau,\non
\eea
For $J_2$, it holds
\bea
J_2&\leq& \int_0^t \|{\rm M}(\vp)\|_{L^\infty(\Omega)}\|\nabla \big(f(\tilde{\varphi})-f(\varphi)\big)\|_{\mathbf{L}^2(\Omega)}\|\nabla\Delta (\tilde{\vp}-\vp)\|_{\mathbf{L}^2(\Omega)} d\tau \non\\
&\leq& C\int_0^t\|\nabla (\tilde{\vp}-\vp)\|_{\mathbf{L}^2(\Omega)}\|\nabla\Delta (\tilde{\vp}-\vp)\|_{\mathbf{L}^2(\Omega)} d\tau\non\\
&& +C\int_0^t \|\tilde{\vp}^2\|_{L^\infty(\Omega)}\|\nabla(\tilde{\vp}-\vp)\|_{\mathbf{L}^2(\Omega)}\|\nabla\Delta (\tilde{\vp}-\vp)\|_{\mathbf{L}^2(\Omega)} d\tau \non\\
&& +C\int_0^t\|\tilde{\vp}+ \vp\|_{L^\infty(\Omega)}\|\nabla \vp\|_{\mathbf{L}^2(\Omega)} \|\tilde{\vp}- \vp\|_{L^\infty(\Omega)}\|\nabla\Delta (\tilde{\vp}-\vp)\|_{\mathbf{L}^2(\Omega)} d\tau \non\\
&\leq& C\int_0^t (1+\|\tilde{\vp}\|^2_{L^\infty(\Omega)})\|\nabla (\tilde{\vp}-\vp)\|_{\mathbf{L}^2(\Omega)}\|\nabla\Delta (\tilde{\vp}-\vp)\|_{\mathbf{L}^2(\Omega)} d\tau \non\\
&&+ C \int_0^t \big(\|\tilde{\vp}\|_{L^\infty(\Omega)}+ \|\vp\|_{L^\infty(\Omega)}\big)\|\nabla\vp\|_{\mathbf{L}^2(\Omega)} \|\nabla\Delta (\tilde{\vp}-\vp)\|_{\mathbf{L}^2(\Omega)}\non\\
&&
\qquad  \times\big(\| \nabla \Delta (\tilde{\vp}-\vp)\|_{\mathbf{L}^2(\Omega)}^\frac14\| \nabla (\tilde{\vp}-\vp)\|_{\mathbf{L}^2(\Omega)}^\frac34+ \| \nabla (\tilde{\vp}-\vp)\|_{\mathbf{L}^2(\Omega)}\big) d\tau \non\\
&\leq& \zeta \int_0^t \|\nabla\Delta (\tilde{\vp}-\vp)\|^2_{\mathbf{L}^2(\Omega)} d\tau \non\\
&&+ C\int_0^t\big(\| \nabla \Delta \tilde{\vp}\|_{\mathbf{L}^2(\Omega)}^2
+\| \nabla \Delta \vp\|_{\mathbf{L}^2(\Omega)}^2+1\big)\|\nabla (\tilde{\vp}-\vp)\|^2_{\mathbf{L}^2(\Omega)} d\tau.\non
\eea

Now we estimate the last term $I_9$,
\bea
I_9&=& \epsilon\int_0^t\int_\Omega  \tilde{\ub}\cdot \nabla (\vp-\tilde{\vp})\Delta (\vp-\tilde{\vp})dxd\tau+\epsilon \int_0^t \int_\Omega \Delta \tilde{\vp}(\tilde{\ub}-\ub)\cdot \nabla(\vp-\tilde{\vp}) dx d\tau\non\\
&\leq& C \int_0^t \|\tilde{\ub}\|_{\mathbf{L}^2(\Omega)}\| \nabla (\vp-\tilde{\vp})\|_{\mathbf{L}^6(\Omega)}\|\Delta (\vp-\tilde{\vp})\|_{L^3(\Omega)} d\tau \non\\
&& +C\int_0^t \| \Delta \tilde{\vp}\|_{L^6(\Omega)}\|\tilde{\ub}-\ub\|_{\mathbf{L}^2(\Omega)}\| \nabla(\vp-\tilde{\vp})\|_{\mathbf{L}^3(\Omega)} d\tau\non\\
&\leq&
  C\int_0^t \|\tilde{\ub}\|_{\mathbf{L}^2(\Omega)} \big(\|\nabla \Delta (\vp-\tilde{\vp})\|_{\mathbf{L}^2(\Omega)}^\frac12\|\nabla (\vp-\tilde{\vp})\|_{\mathbf{L}^2(\Omega)}^\frac12+\|\nabla (\vp-\tilde{\vp})\|_{\mathbf{L}^2(\Omega)}\big)
  \non\\
  &&\quad \times
  \big(\|\nabla \Delta (\vp-\tilde{\vp})\|_{\mathbf{L}^2(\Omega)}^\frac34\|\nabla (\vp-\tilde{\vp})\|_{\mathbf{L}^2(\Omega)}^\frac14+\|\nabla (\vp-\tilde{\vp})\|_{\mathbf{L}^2(\Omega)}\big) d\tau \non\\
  && + C\int_0^t (\| \nabla \Delta \tilde{\vp}\|_{L^2(\Omega)}+\| \nabla \tilde{\vp}\|_{L^2(\Omega)})\|\tilde{\ub}-\ub\|_{\mathbf{L}^2(\Omega)}
  \non\\
  && \quad \times \big(\|\nabla \Delta (\vp-\tilde{\vp})\|_{\mathbf{L}^2(\Omega)}^\frac14\|\nabla (\vp-\tilde{\vp})\|_{\mathbf{L}^2(\Omega)}^\frac34+\|\nabla (\vp-\tilde{\vp})\|_{\mathbf{L}^2(\Omega)}\big)  d\tau\non\\
  &\leq& \zeta \int_0^t \|\nabla\Delta (\tilde{\vp}-\vp)\|^2_{\mathbf{L}^2(\Omega)} d\tau +\zeta \int_0^t  \|\tilde{\ub}-\ub\|_{\mathbf{L}^2(\Omega)}^2 d\tau\non\\
  && +C\int_0^t \big(\|\tilde{\ub}\|_{\mathbf{L}^2(\Omega)}^\frac83+\| \nabla\Delta \tilde{\vp}\|_{L^2(\Omega)}^\frac83+1\big) \|\nabla (\tilde{\vp}-\vp)\|^2_{\mathbf{L}^2(\Omega)} d\tau.
\eea

Combining the above estimates, using the equivalent norm $\|\ub\|_{\mathbf{Z}}$ given by \eqref{equinorm} in Lemma \ref{equinorml} and the assumptions (A1)--(A3), by taking $\zeta>0$ sufficiently small, we deduce that
\bea
&& \varpi\|(\tilde{\ub}_c-\ub_c)(t)\|^2_{\mathbf{L}^2(\Omega_c)}+\epsilon\|\nabla (\tilde{\vp}-\vp)(t)\|^2_{\mathbf{L}^2(\Omega)}\non\\
&&\ +\gamma_1\int_0^t \big(\|(\tilde{\ub}-\ub)(\tau)\|_{\mathbf{Z}}^2+\| \nabla \Delta (\tilde{\vp}-\vp)(\tau)\|_{\mathbf{L}^2(\Omega)}^2\big) d\tau\non\\
&\leq& \gamma_2\int_0^t h(\tau)\|\nabla (\tilde{\vp}-\vp)(\tau)\|_{\mathbf{L}^2(\Omega)}^2 d\tau, \label{gron}
\eea
where
$$
h(t)=\|\tilde{\ub}(t)\|_{\mathbf{Z}}^\frac83+ \|\nabla \Delta \tilde{\vp}(t)\|_{\mathbf{L}^2(\Omega)}^\frac83+\| \nabla \Delta \vp(t)\|_{\mathbf{L}^2(\Omega)}^2+1,
$$
and the constants $\gamma_1, \gamma_2>0$ may depend on the initial energy $\mathcal{E}(0)$  as well as the coefficients of the CHSD system.

Since by our assumption $(\ub_c, \vp)|_{t=0}=(\mathbf{0}, 0)$ and $h(t)\in L^1(0,T)$, then it follows from \eqref{gron} and the Gronwall inequality that for $t\in [0,T]$,
\be
\varpi\|(\tilde{\ub}_c-\ub_c)(t)\|^2_{\mathbf{L}^2(\Omega_c)}+\epsilon\|\nabla (\tilde{\vp}-\vp)(t)\|^2_{\mathbf{L}^2(\Omega)}=0 \label{uu1}
\ee
and then
\be
\int_0^T \|(\tilde{\ub}-\ub)(t)\|_{\mathbf{Z}}^2dt =0.\label{uu2}
\ee
Recalling the fact $\int_\Omega (\tilde{\vp}-\vp) dx=0$ for $t\in[0,T]$, by the Poincar\'e inequality and the definition of the norm $\|\cdot\|_{\mathbf{Z}}$ (see \eqref{equinorm}), we infer that
\be
(\ub_c, \ub_m, \vp)=(\tilde{\ub}_c, \tilde{\ub}_m, \tilde{\vp}).\label{uu3}
\ee

Finally, we remark that for the case of $\varpi=0$, one can proceed as above and conclude \eqref{uu1}, \eqref{uu2} with $\varpi=0$ in \eqref{uu1}, which again yield the uniqueness result \eqref{uu3}.

The proof of Theorem \ref{thmuni} is complete.

\section*{Acknowledgement.}
Wang and Han acknowledge the  support of  NSF (DMS1312701) and a Planning Grant from FSU.
Wu was partially supported by NSF of China 11371098 and ``Zhuo Xue" program of Fudan University.

\itemsep=0pt
\bibliography{newbib}

\begin{thebibliography}{10}

\bibitem{Ab09}
Helmut Abels.
\newblock Existence of weak solutions for a diffuse interface model for
  viscous, incompressible fluids with general densities.
\newblock {\em Comm. Math. Phys.}, 289(1):45--73, 2009.

\bibitem{Abels2009}
Helmut Abels.
\newblock On a diffuse interface model for two-phase flows of viscous,
  incompressible fluids with matched densities.
\newblock {\em Arch. Ration. Mech. Anal.}, 194(2):463--506, 2009.

\bibitem{AbLe2012}
Helmut Abels and Daniel Lengeler.
\newblock On sharp interface limits for diffuse interface models for two-phase
  flows.
\newblock {\em arXiv:1212.5582}, 2012.

\bibitem{BDQ10}
Lori Badea, Marco Discacciati, and Alfio Quarteroni.
\newblock Numerical analysis of the {N}avier-{S}tokes/{D}arcy coupling.
\newblock {\em Numer. Math.}, 115(2):195--227, 2010.

\bibitem{Bear}
J.~Bear.
\newblock {\em Dynamics of Fluids in Porous Media}.
\newblock Courier Dover Publications, 1988.

\bibitem{B-J}
G.S. Beavers and D.D. Joseph.
\newblock Boundary conditions at a naturally permeable wall.
\newblock {\em J. Fluid Mech.}, 30(1):197--207, 1967.

\bibitem{Boyer1999}
Franck Boyer.
\newblock Mathematical study of multi-phase flow under shear through order
  parameter formulation.
\newblock {\em Asymptot. Anal.}, 20(2):175--212, 1999.

\bibitem{BG}
H.~Br{\'e}zis and T.~Gallouet.
\newblock Nonlinear {S}chr\"odinger evolution equations.
\newblock {\em Nonlinear Anal.}, 4(4):677--681, 1980.

\bibitem{CGHW10}
Yanzhao Cao, Max Gunzburger, Fei Hua, and Xiaoming Wang.
\newblock Coupled {S}tokes-{D}arcy model with {B}eavers-{J}oseph interface
  boundary condition.
\newblock {\em Commun. Math. Sci.}, 8(1):1--25, 2010.

\bibitem{CGHW11}
Yanzhao Cao, Max Gunzburger, Fei Hua, and Xiaoming Wang.
\newblock Analysis and finite element approximation of a coupled, continuum
  pipe-flow/{D}arcy model for flow in porous media with embedded conduits.
\newblock {\em Numer. Methods Partial Differential Equations},
  27(5):1242--1252, 2011.

\bibitem{CR08}
A.~{\c{C}}e{\c{s}}melio{\u{g}}lu and B.~Rivi{\`e}re.
\newblock Analysis of time-dependent {N}avier-{S}tokes flow coupled with
  {D}arcy flow.
\newblock {\em J. Numer. Math.}, 16(4):249--280, 2008.

\bibitem{CGR13}
Ay{\c{c}}{\i}l {\c{C}}e{\c{s}}melio{\u{g}}lu, Vivette Girault, and B{\'e}atrice
  Rivi{\`e}re.
\newblock Time-dependent coupling of navier-stokes and darcy flows.
\newblock {\em ESAIM: M2AN}, 47:539--554, 2013.

\bibitem{CR12}
Ay{\c{c}}{\i}l {\c{C}}e{\c{s}}melio{\u{g}}lu and B{\'e}atrice Rivi{\`e}re.
\newblock Existence of a weak solution for the fully coupled
  {N}avier-{S}tokes/{D}arcy-transport problem.
\newblock {\em J. Differential Equations}, 252(7):4138--4175, 2012.

\bibitem{CGW10}
Nan Chen, Max Gunzburger, and Xiaoming Wang.
\newblock Asymptotic analysis of the differences between the {S}tokes-{D}arcy
  system with different interface conditions and the {S}tokes-{B}rinkman
  system.
\newblock {\em J. Math. Anal. Appl.}, 368(2):658--676, 2010.

\bibitem{ChGHW11}
Wenbin Chen, Max Gunzburger, Fei Hua, and Xiaoming Wang.
\newblock A parallel {R}obin-{R}obin domain decomposition method for the
  {S}tokes-{D}arcy system.
\newblock {\em SIAM J. Numer. Anal.}, 49(3):1064--1084, 2011.

\bibitem{CGSW2013}
Wenbin Chen, Max Gunzburger, Dong Sun, and Xiaoming Wang.
\newblock Efficient and long-time accurate second-order methods for the
  {S}tokes-{D}arcy system.
\newblock {\em SIAM J. Numer. Anal.}, 51(5):2563--2584, 2013.

\bibitem{CR09}
Prince Chidyagwai and B{\'e}atrice Rivi{\`e}re.
\newblock On the solution of the coupled {N}avier-{S}tokes and {D}arcy
  equations.
\newblock {\em Comput. Methods Appl. Mech. Engrg.}, 198(47-48):3806--3820,
  2009.

\bibitem{CKRS}
Pierluigi Colli, Pavel Krej{\v{c}}{\'{\i}}, Elisabetta Rocca, and J{\"u}rgen
  Sprekels.
\newblock Nonlinear evolution inclusions arising from phase change models.
\newblock {\em Czechoslovak Math. J.}, 57(132)(4):1067--1098, 2007.

\bibitem{De}
Klaus Deimling.
\newblock {\em Nonlinear Functional Analysis}.
\newblock Springer-Verlag, Berlin, 1985.

\bibitem{DFW2014}
Amanda~E. Diegel, Xiaobing Feng, and Steven~M. Wise.
\newblock Analysis of a mixed finite element method for a
  {C}ahn-{H}illiard-{D}arcy-{S}tokes system.
\newblock {\em arXiv:1312.1313v3}, 2014.

\bibitem{DiQu2003}
M.~Discacciati and A.~Quarteroni.
\newblock Analysis of a domain decomposition method for the coupling of the
  stokes and darcy equations.
\newblock In {\em Numerical Mathematics and Advanced Applications}, volume 320,
  pages 3--20. Springer, Milan, 2003.

\bibitem{DMQ2002}
Marco Discacciati, Edie Miglio, and Alfio Quarteroni.
\newblock Mathematical and numerical models for coupling surface and
  groundwater flows.
\newblock {\em Appl. Numer. Math.}, 43(1-2):57--74, 2002.

\bibitem{DiQu2009}
Marco Discacciati and Alfio Quarteroni.
\newblock Navier-{S}tokes/{D}arcy coupling: modeling, analysis, and numerical
  approximation.
\newblock {\em Rev. Mat. Complut.}, 22(2):315--426, 2009.

\bibitem{Eyre}
David~J. Eyre.
\newblock Unconditionally gradient stable time marching the {C}ahn-{H}illiard
  equation.
\newblock In {\em Computational and mathematical models of microstructural
  evolution ({S}an {F}rancisco, {CA}, 1998)}, volume 529 of {\em Mater. Res.
  Soc. Sympos. Proc.}, pages 39--46. MRS, Warrendale, PA, 1998.

\bibitem{FeWi2012}
Xiaobing Feng and Steven Wise.
\newblock Analysis of a {D}arcy-{C}ahn-{H}illiard diffuse interface model for
  the {H}ele-{S}haw flow and its fully discrete finite element approximation.
\newblock {\em SIAM J. Numer. Anal.}, 50(3):1320--1343, 2012.

\bibitem{GaGr2010}
Ciprian~G. Gal and Maurizio Grasselli.
\newblock Asymptotic behavior of a {C}ahn-{H}illiard-{N}avier-{S}tokes system
  in 2{D}.
\newblock {\em Ann. Inst. H. Poincar\'e Anal. Non Lin\'eaire}, 27(1):401--436,
  2010.

\bibitem{GiRa1986}
Vivette Girault and Pierre-Arnaud Raviart.
\newblock {\em Finite Element Methods for {N}avier-{S}tokes Equations},
  volume~5 of {\em Springer Series in Computational Mathematics}.
\newblock Springer-Verlag, Berlin, 1986.

\bibitem{Grisvard1985}
P.~Grisvard.
\newblock {\em Elliptic Problems in Nonsmooth Domains}, volume~24 of {\em
  Monographs and Studies in Mathematics}.
\newblock Pitman (Advanced Publishing Program), Boston, MA, 1985.

\bibitem{HSW13}
Daozhi Han, Dong Sun, and Xiaoming Wang.
\newblock Two phase flows in karstic geometry.
\newblock {\em Math. Methods Appl. Sci.}, 2013.
\newblock In press.

\bibitem{HaWa2014}
Daozhi Han and Xiaoming Wang.
\newblock A second order in time, uniquely solvable, unconditionally stable
  numerical scheme for {C}ahn-{H}illiard-{N}avier-{S}tokes equation.
\newblock 2014.
\newblock in preparation.

\bibitem{Hor95}
C.~O. Horgan.
\newblock Korn's inequalities and their applications in continuum mechanics.
\newblock {\em SIAM Rev.}, 37(4):491--511, 1995.

\bibitem{J-M}
Willi J{\"a}ger and Andro Mikeli{\'c}.
\newblock On the interface boundary condition of {B}eavers, {J}oseph, and
  {S}affman.
\newblock {\em SIAM J. Appl. Math.}, 60(4):1111--1127, 2000.

\bibitem{Jones}
IP~Jones.
\newblock Low reynolds-number flow past a porous spherical shell.
\newblock {\em Proceedings of the Cambridge Philosophical Society},
  73(JAN):231--238, 1973.

\bibitem{LSY2002}
William~J. Layton, Friedhelm Schieweck, and Ivan Yotov.
\newblock Coupling fluid flow with porous media flow.
\newblock {\em SIAM J. Numer. Anal.}, 40(6):2195--2218, 2002.

\bibitem{Lions1969}
J.-L. Lions.
\newblock {\em Quelques Methodes de Resolution des Provl\'emes aux Limites non
  Lin\'eaires}.
\newblock Dunod, Paris, 1969.

\bibitem{L-M}
J.-L. Lions and E.~Magenes.
\newblock {\em Non-homogeneous Boundary Value Problems and Applications. {V}ol.
  {I}}.
\newblock Springer-Verlag, New York, 1972.
\newblock Translated from the French by P. Kenneth, Die Grundlehren der
  mathematischen Wissenschaften, Band 181.

\bibitem{LS03}
Chun Liu and Jie Shen.
\newblock A phase field model for the mixture of two incompressible fluids and
  its approximation by a {F}ourier-spectral method.
\newblock {\em Phys. D}, 179(3-4):211--228, 2003.

\bibitem{LT98}
J.~Lowengrub and L.~Truskinovsky.
\newblock Quasi-incompressible {C}ahn-{H}illiard fluids and topological
  transitions.
\newblock {\em R. Soc. Lond. Proc. Ser. A Math. Phys. Eng. Sci.},
  454(1978):2617--2654, 1998.

\bibitem{LTZ2013}
John Lowengrub, Edriss Titi, and Kun Zhao.
\newblock Analysis of a mixture model of tumor growth.
\newblock {\em Euro. J. Appl. Math.}, 24(5):691--734, 2013.

\bibitem{Saffman}
P.~G. Saffman.
\newblock On the boundary condition at the interface of a porous medium.
\newblock {\em Stud. in Appl. Math.}, 1:93--101, 1971.

\bibitem{Showalter1997}
R.~E. Showalter.
\newblock {\em Monotone Operators in {B}anach Space and Nonlinear Partial
  Differential Equations}, volume~49 of {\em Mathematical Surveys and
  Monographs}.
\newblock American Mathematical Society, Providence, RI, 1997.

\bibitem{Si85}
Jacques Simon.
\newblock Compact sets in the space {$L^p(0,T;B)$}.
\newblock {\em Ann. Mat. Pura Appl. (4)}, 146:65--96, 1987.

\bibitem{Temam1977}
Roger Temam.
\newblock {\em Navier-{S}tokes Equations. {T}heory and {N}umerical {A}nalysis},
  volume~2 of {\em Studies in Mathematics and its Applications}.
\newblock North-Holland, Amsterdam-New York-Oxford, 1977.

\bibitem{Tri}
H.~Triebel.
\newblock {\em Interpolation Theory, Function Spaces, Differential Operators}.
\newblock North-Holland, Amsterdam, 1978.

\bibitem{WaWu2012}
Xiaoming Wang and Hao Wu.
\newblock Long-time behavior for the {H}ele-{S}haw-{C}ahn-{H}illiard system.
\newblock {\em Asymptot. Anal.}, 78(4):217--245, 2012.

\bibitem{WaZh2013}
Xiaoming Wang and Zhifei Zhang.
\newblock Well-posedness of the {H}ele-{S}haw-{C}ahn-{H}illiard system.
\newblock {\em Ann. Inst. H. Poincar\'e Anal. Non Lin\'eaire}, 30(3):367--384,
  2013.

\bibitem{Wise10}
S.~M. Wise.
\newblock Unconditionally stable finite difference, nonlinear multigrid
  simulation of the {C}ahn-{H}illiard-{H}ele-{S}haw system of equations.
\newblock {\em J. Sci. Comput.}, 44(1):38--68, 2010.

\bibitem{ZWH2009}
Liyun Zhao, Hao Wu, and Haiyang Huang.
\newblock Convergence to equilibrium for a phase-field model for the mixture of
  two viscous incompressible fluids.
\newblock {\em Commun. Math. Sci.}, 7(4):939--962, 2009.

\end{thebibliography}
\bibliographystyle{plain}

\end{document}